\documentclass[11pt,letterpaper]{amsart}

\usepackage{epsfig,amsmath,amsthm,amssymb,graphicx,hyperref}
\usepackage[top=1.25in, bottom=1.25in, left=1.25in, right=1.25in]{geometry}
\usepackage{caption}

\newtheorem{proposition}{Proposition}[section]
\newtheorem{theorem}[proposition]{Theorem}
\newtheorem{conjecture}[proposition]{Conjecture}

\newtheorem*{theorem*}{Theorem}

\newtheorem*{thm_2.1}{Theorem 2.1}
\newtheorem*{thm_3.1}{Theorem 3.1}
\newtheorem*{thm_4.1}{Theorem 4.1}

\newtheorem{corollary}[proposition]{Corollary}

\theoremstyle{definition}
\newtheorem{definition}[proposition]{Definition}
\newtheorem{remark}[proposition]{Remark}

\newcommand{\K}{\text{K}}
\newcommand{\M}{\text{M}}
\newcommand{\A}{\text{A}}
\newcommand{\N}{\text{N(A)}}
\newcommand{\ND}{\text{N}(\text{D}^2)}
\newcommand{\D}{{\text{D}}^2}
\newcommand{\B}{{\text{B}}^4}
\newcommand{\C}{{\text{S}}^1}
\newcommand{\CCC}{{\text{S}}^3}
\newcommand{\F}{\text{F}}
\newcommand{\R}{\text{R}}

\begin{document}
\title[A Construction of Slice Knots via Annulus Modifications]{A Construction of Slice Knots via Annulus Modifications}

\author{JungHwan Park}
\address{Department of Mathematics, Rice University MS-136\\
6100 Main St. P.O. Box 1892\\
Houston, TX 77251-1892}
\email{jp35@rice.edu}

\date{\today}


\begin{abstract}
We define an operation on homology $\B$ which we call an $n$-twist annulus modification. We give a new construction of smoothly slice knots and exotically slice knots via $n$-twist annulus modifications. As an application, we present a new example of a smoothly slice knot with non-slice derivatives. Such examples were first discovered by Cochran and Davis. Also, we relate $n$-twist annulus modifications to $n$-fold annulus twists which was first introduced by Osoinach, then has been used by Abe and Tange to construct smoothly slice knots. Furthermore we consider $n$-twist annulus modifications in more general setting to show that any exotically slice knot can be obtained by the image of the unknot in the boundary of a smooth $4$-manifold homeomorphic to $\B$ after an annulus modification.
\end{abstract}

\maketitle
\section{Introduction}

We begin with some definitions. A knot $\K$ is an embedding of an oriented $\C$ into $\CCC$ and a link $\text{L}$ is an embedding of a disjoint collection of oriented $\C$ into $\CCC$. If $\K$ bounds a smoothly embedded $2$-disk $\D$ in the standard $\B$, we call $\K$ a smoothly slice knot and $\D$ a smooth slice disk. Moreover, if $\K$ bounds a smoothly embedded $2$-disk $\D$ in a smooth oriented  $4$-manifold $\M$ that is homeomorphic to the standard $\B$ but not necessarily diffeomorphic to the standard $\B$, we call $\K$ exotically slice in $\M$. Also, if a smoothly slice knot $\K$ bounds a smoothly embedded $2$-disk $\D$ in the standard $\B$ where there are no local maximum of the height function restricted to $\D$, we call $\K$ a ribbon knot. A link $\text{L}$ is a smoothly slice link if each component of $\text{L}$ bounds a smoothly embedded disjoint $2$-disk $\D$ in the standard $\B$.

In this paper we will define an operation on homology $\B$ which we call an $n$-twist annulus modification in Section $2$. Then we will fix a slice knot and use this modification under a certain condition (see Section $2$), to obtain an infinite family of exotically slice knots. The basic idea for this condition is that there exist a smooth proper embedding of an annulus disjoint from a smooth slice disk (see Figure $1$), which is called $l$-nice. The technique we use here is similar to the technique which was used in \cite{CD14} to construct a smoothly slice knot with non-slice derivatives (see Section $1.1$).
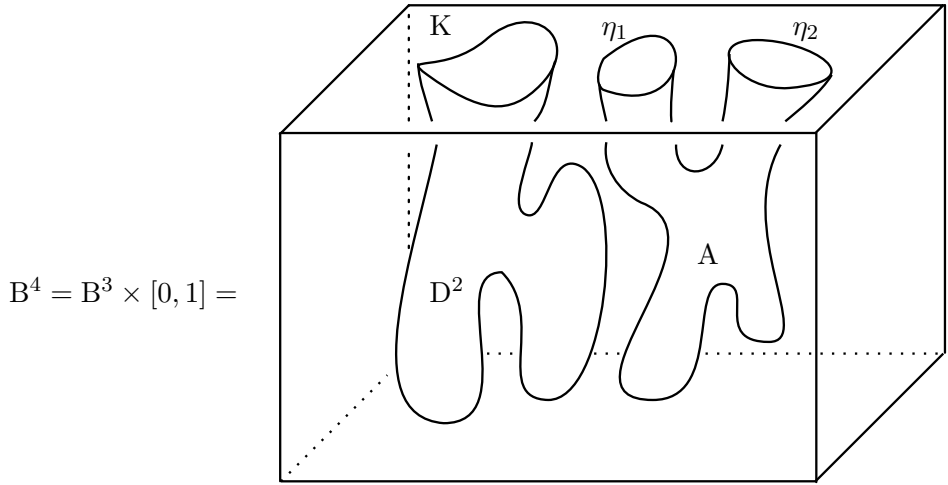
\begin{figure}[h]
\centering
\includegraphics[width=4.5in]{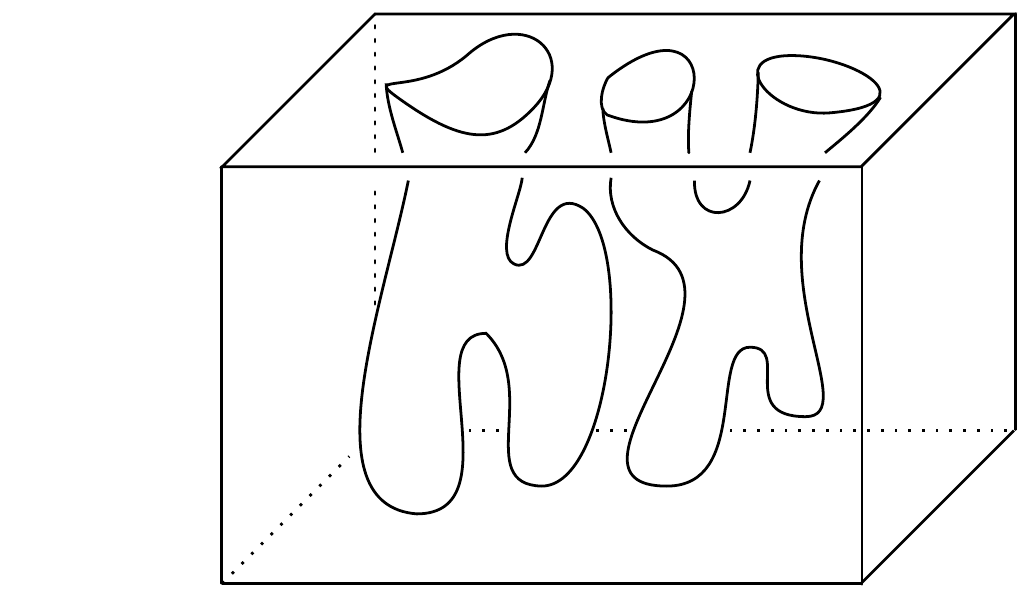}
  \put(-5,1){ $\B = \text{B}^3 \times [0 , 1] =$}
    \put(-2.8,2.4){ $\K$}
        \put(-1.9,2.4){ $\eta_1$}
            \put(-0.9,2.4){ $\eta_2$}
 \put(-2.8,1){ $\D$}
 \put(-1.4,1.2){ $\A$}
  
\caption{Schematic picture of the condition $l$-nice.}
\end{figure}

\begin{thm_2.1}\label{Main}Let $K$ be a smoothly slice knot bounding a smooth disk $\D$ in the standard $\B$, and let $(\{ \eta_1, \eta_2 \}, \phi_A)$ be $l$-nice. Then $\frac{nl+1}{n}$ Dehn surgery on $\eta_1$ followed by $\frac{nl-1}{n}$ Dehn surgery on $\eta_2$ will produce an exotically slice knot $\K_{(\phi_\A,n)} \subseteq \CCC$ for any integer $n$, where $\K_{(\phi_\A,n)}$ is the image of $\K$ in the new $3$-manifold $\CCC$.\end{thm_2.1}

If we restrict our condition further (see Section $3$), which is called $l$-standard, we can use annulus modifications to obtain a infinite family of smoothly slice knots. The basic idea for this condition is that the smooth proper embedding of an annulus is isotopic to the standard annulus which is $\A_l$ in Figure $7$.

\begin{thm_3.1}\label{Diffeo}Let $\K$ be a smoothly slice knot bounding a smooth disk $\D$ in the standard $\B$, and let $(\{ \eta_1, \eta_2 \}, \phi_A)$ be $l$-standard. Then $\frac{nl+1}{n}$ Dehn surgery on $\eta_1$ followed by $\frac{nl-1}{n}$ Dehn surgery on $\eta_2$ will produce a smoothly slice knot $\K_{(\phi_\A,n)}\subseteq \CCC$ for any integer $n$, where $\K_{(\phi_\A,n)}$ is the image of $\K$ in the new $3$-manifold $\CCC$.\end{thm_3.1}

We have two applications of these theorems.

\subsection{Application 1 : An example of a slice knot with non-slice derivatives}
Recall that any knot in $\CCC$ bounds a Seifert Surface $\F$. From $\F$, we can define a Seifert form $\beta_\F : H_1(F) \times H_1(F) \rightarrow \mathbb{Z}$, which is defined by $\beta_\F( [x], [y]) = lk(x,y^+)$, where $x$ is union of simple closed curves on F which represents $[x]$, $y^+$ is positive push off of union of simple closed curves on F which represents $[y]$, and $lk$ denotes linking number. It was proven by Levine \cite[Lemma2]{Le69} that if $\K$ is a smoothly slice knot then $\beta_\F$ is metabolic for any Seifert surface $\F$ for $\K$, i.e.\ there exists $H=\mathbb{Z}^{\frac{1}{2} rank H_1(F)}$ a direct summand of $H_1(F)$, such that $\beta_F$ vanishes on $H$. We call a knot algebraically slice knot if it has metabolic Seifert form. Then a link $\{\gamma_1, \gamma_2, \cdots, \gamma_{\frac{1}{2} rank H_1(F)}\}$ disjointly embedded in a surface $\F$ where homology class forms a basis for $H$ is called a derivative of $\K$ (see Figure $2$). Notice that we can define a derivative of a knot for any algebraically slice knots. 

\begin{figure}[h]
\centering
\includegraphics[width=2.5in]{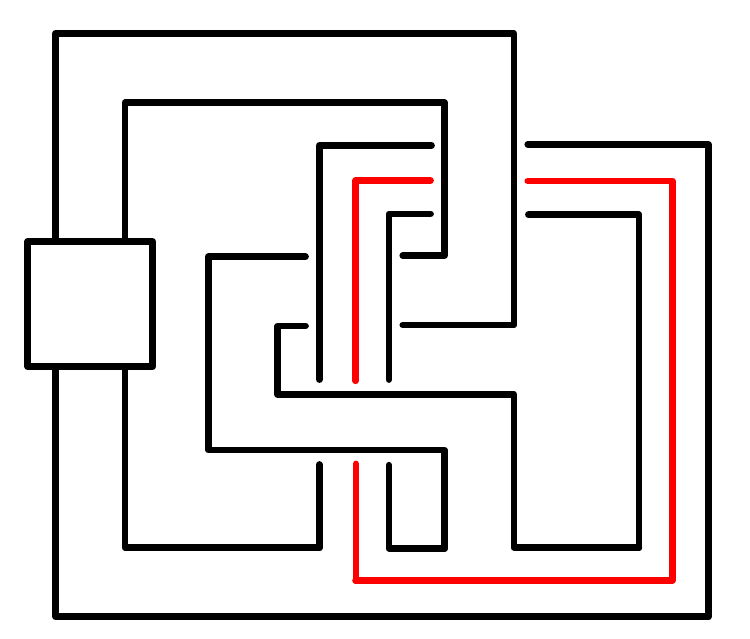}
  \put(-2.3,1.07){$-1$}
\caption{A derivative of $6_{1}$ knot}
\end{figure}

If a knot $\K$ has a derivative which is a smoothly slice link, then $\K$ is smoothly slice (see Figure $3$). A natural question is whether the converse holds. This was asked by Kauffman in 1982, for genus $1$ knots.

\begin{figure}[h]
\centering
\includegraphics[width=2.5in]{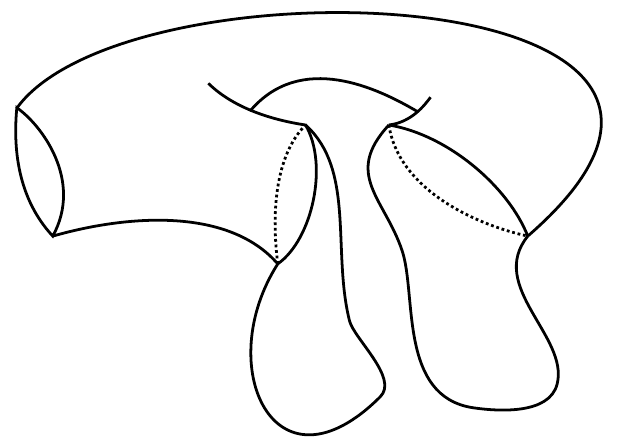}
\caption{Surgery on derivative gives us smoothly slice disk}
\end{figure}

\begin{conjecture} \cite[$\text{N}1.52$]{Ka87,kirbylist} If $\K$ is a smoothly slice knot and $\F$ is a genus $1$ Seifert surface for $\K$ then there exists an essential simple closed curve $d$ on $\F$ such that $lk(d,d^+)=0$ and $d$ is a slice knot. \end{conjecture}

A lot of evidence which supported this conjecture were found by Casson, Cooper, Gilmer, Gordon, Livingsont, Litherland, Cochran, Orr, Teichner, Harvey, Leidy and others \cite{Gi83,Li84,Gi93,GL92,GL92',GL13,CHL10,COT04}. However, the conjecture was false: Cochran and Davis recently constructed a smoothly slice knot $\K$ where neither of its derivatives is smoothly slice in \cite{CD14}. Surprisingly both of the derivatives have non-zero Arf invariant. In particular, they are not algebraically slice. In this paper we present a different example of a smoothly slice knot with non-slice derivatives. We obtain this example by using an $n$-twist annulus modification.

\begin{thm_4.1}\label{Main}Let $\R_1$ be a knot described in Section $4$. Then $\R_1$ is a smoothly slice knot with non-slice derivatives.\end{thm_4.1}

The Kauffman's conjecture can be easily generalized to higher genus as follows.

\begin{conjecture} If $\K$ is a smoothly slice knot and $\F$ is a genus $g$ Seifert surface for $\K$ then there exists a link $\text{L} = \{\gamma_1, \gamma_2, \cdots, \gamma_g \}$ on $\F$ such that $\text{L}$ is a derivative $\K$ and $\text{L}$ is a smoothly slice link.\end{conjecture}

It is still an open problem whether Kauffman's conjecture is true for knots with genus greater than $1$.

\subsection{Application 2 : Application related to an annulus twist}
An Annulus twist (see Section $5$) is an operation on $\CCC$ which was used in \cite{Os06}. Osoinach used annulus twists to produce $3$-manifolds that can be obtained by same coefficient Dehn surgery on an infinite family of distinct knots. In particular, Osoinach showed that if a knot $\K$ has an orientation preserving annulus presentation (see Section $5$), $0$-surgery on $\K_n$, where $\K_n$ is a knot obtained by an $n$-fold annulus twist on $\K$, is diffeomorphic to $0$-surgery on $\K$. Thus, if $\K$ is a smoothly slice knot with an orientation preserving annulus presentation, $\K_n$ is exotically slice for any integer $n$ \cite[Proposition $1.2$]{CFHeHo13}. This was also pointed out in \cite{AJOT13}, where they use annulus twists to produce an infinite family of distinct framed knots which represents a diffeomorphic $4$-manifold. In this paper, we will reprove the statement with slightly stronger assumptions, using $n$-twist annulus modifications. More precisely, we will show that $\K_n$ is exotically slice if $\K$ is a ribbon knot.

In fact, in \cite{AT14} Abe and Tange showed that if $\K$ is a ribbon knot with an annulus presentation with $+1$ or $-1$ framing (see Section $5$), $\K_n$ is smoothly slice. In this paper, we use $n$-twist annulus modifications to show this statement is true for a very specific case, namely when $\K$ is $8_{20}$ knot. Also, in \cite{AT14} they show that an $n$-fold annulus twist on $8_{20}$ are ribbon knot when $n\ge 0$, but it is still not known whether other slice knots obtained by an annulus twist are ribbon knots.

$ $

In the last section we consider $n$-twist annulus modifications on general annuli. More precisely, we no longer require the link $\{ \eta_1, \eta_2 \}$ to be isotopic to $L_l$ (see Figure $5$), which is one of the requirement for $(\{ \eta_1, \eta_2 \}, \phi_A)$ to be either $l$-nice or $l$-standard. By using these general annuli, we show that any exotically slice knots can be obtained by the image of the unknot in the boundary of a smooth $4$-manifold homeomorphic to $\B$ after an annulus modification. Notice that this tells us that for any two smoothly slice knots $\K_1$ and $\K_2$, it is possible to get from one to the other by performing two annulus modifications, if the $4$-dimensional smooth Poincar\'e Conjecture is true.

\subsection{Acknowledgements} The author would like to thank his advisors Tim Cochran and Shelly Harvey, and also Arunima Ray, and Christopher Davis for their helpful discussions.
\section{The Technique}\label{Construction}

In this section, we discuss the technique for constructing new slice knots from a fixed smoothly slice knot and slice disk. 

Let $\M$ be a smooth compact $4$-manifold with non empty boundary, and assume $\M$ is a integer homology four ball. Let $\phi_\A : \C \times [0, 1] \hookrightarrow \M$ be a smooth proper embedding of an annulus with $\text{Im}(\left. \phi_\A \right |_{\C \times \{ 0 \}})= \eta_1$, $\text{Im}(\left. \phi_A \right |_{\C \times \{ 1 \}})= \eta_2$, and $\text{Im}(\phi_\A)= \A$. Further assume $\{ \eta_1, \eta_2 \}$ is contained in some three ball in $\partial \M$ so it makes sense to talk about the linking number of $\eta_1$ and $\eta_2$. Then let $l=lk(\eta_1, \eta_2)$ and $n$ be any integer. Note that, $\phi_\A$ can be extended to a smooth proper embedding $\phi_{\N} : \C \times \D \times [0, 1] \hookrightarrow \M$ where $\text{Im}{(\phi_\N)}= \N$ where $\N$ is a tubular neighborhood of $\A$. Notice also that we can choose $\phi_\N$ so that $\C \times \{ 1 \} \times \{ 0 \}$ is identified with preferred longitude of $\eta_1$. Let $\lambda_1$ and $\lambda_2$ be preferred longitudes for $\eta_1$ and $\eta_2$ respectively, and let $\mu_1$ and $\mu_2$ be meridians of $\eta_1$ and $\eta_2$ respectively. Then $\phi_\N$ gives the following identifications:

\begin{itemize}
\item $\C \times \{ 1 \} \times \{ 0 \} \subseteq N$ is identified with $\lambda_1$.
\item $\{ 1 \} \times \partial \D \times \{ 0 \} \subseteq N$ is identified with $\mu_1$.
\item $\C \times \{ 1 \} \times \{ 1 \} \subseteq N$ is identified with $\lambda_2 + 2l\mu_2$.
\item $\{ 1 \} \times \partial \D \times \{ 1 \} \subseteq N$ is identified with $-\mu_2$.
\end{itemize}

Let $Aut(\C \times \C)$ be the set of isotopy classes of diffeomorphisms from $\C \times \C$ to itself. Recall there is a bijective correspondence between $Aut(\C \times \C)$ and $GL(2,\mathbb{Z})$ \cite{Ro90}. Then let $\rho_n$ be the element in $Aut(\C \times \C)$ which corresponds to 
$\begin{pmatrix}
  1 & n \\
  l & nl+1 
 \end{pmatrix}$, that is, $\rho_n$ sends the longitude to the longitude plus $l$ times the meridian and the meridian to $n$ times the longitude plus $nl+1$ times the meridian.
 
Let $\psi_n$ be the diffeomorphism defined as follows: 
$$\psi_n : \C \times \partial \D \times [0 , 1] \to \C \times \partial \D \times [0 , 1]$$ $$(x,y,t) \rightarrow (\rho_n(x,y),t) $$ where $x \in \C$, $y \in \partial \D$ and $t \in [0 , 1]$.

Using $\left. \phi_n \right |_{\C \times \partial \D \times [0 , 1]}$ and $\psi_n$, we define our modification on $\M$. Let $f_n := \left. \phi_\A \right |_{\C \times \partial \D \times [0 , 1]} \circ \psi_n$; note that $f_n$ is a diffeomorphism from $\C \times \partial \D \times [0 , 1]$ to itself. Then let $\M_{(\phi_\A,n)}=(\M - \N) \cup_{f_n} \C \times [0, 1] \times \D$. We will call $\M_{(\phi_\A,n)}$ \textit{the $n$-twist annulus modification on $\M$ at $\phi_\A$}. (This can be thought as doing Dehn surgery along the interval.)

We will perform annulus modifications on $\B$ to get new smoothly slice knots and new exotically slice knots.
First, we describe the basic idea. Fix a smoothly slice knot $\K$ with a smooth slice disk $\D$, in the 4-ball $\B$ from now on. Let $\text{N}$ be an smoothly embedded 4-manifold in $\B \setminus \ND$ (see Figure $4$). Let $\widetilde{B}$ be a new manifold obtained by taking out $\text{N}$ and glueing it back differently to the complement. Let $\widetilde{K} \subseteq \partial \widetilde{B}$ be the image of $\K$ in the modified manifold. Since $\widetilde{K}$ bounds a smoothly embedded disk in $\widetilde{B}$, if $\widetilde{B}$ is homeomorphic to $\B$, then the resulting new knot $\widetilde{K} \subseteq \partial \widetilde{B}$ is exotically slice, and if  $\widetilde{B}$ is diffeomorphic to $\B$, then the resulting new knot $\widetilde{K} \subseteq \partial \widetilde{B}$ is smoothly slice. In our case $N$ is going to be a tubular neighborhood of a particular proper smooth embedding of an annulus.

\begin{figure}[h]
\centering
\includegraphics[width=5in]{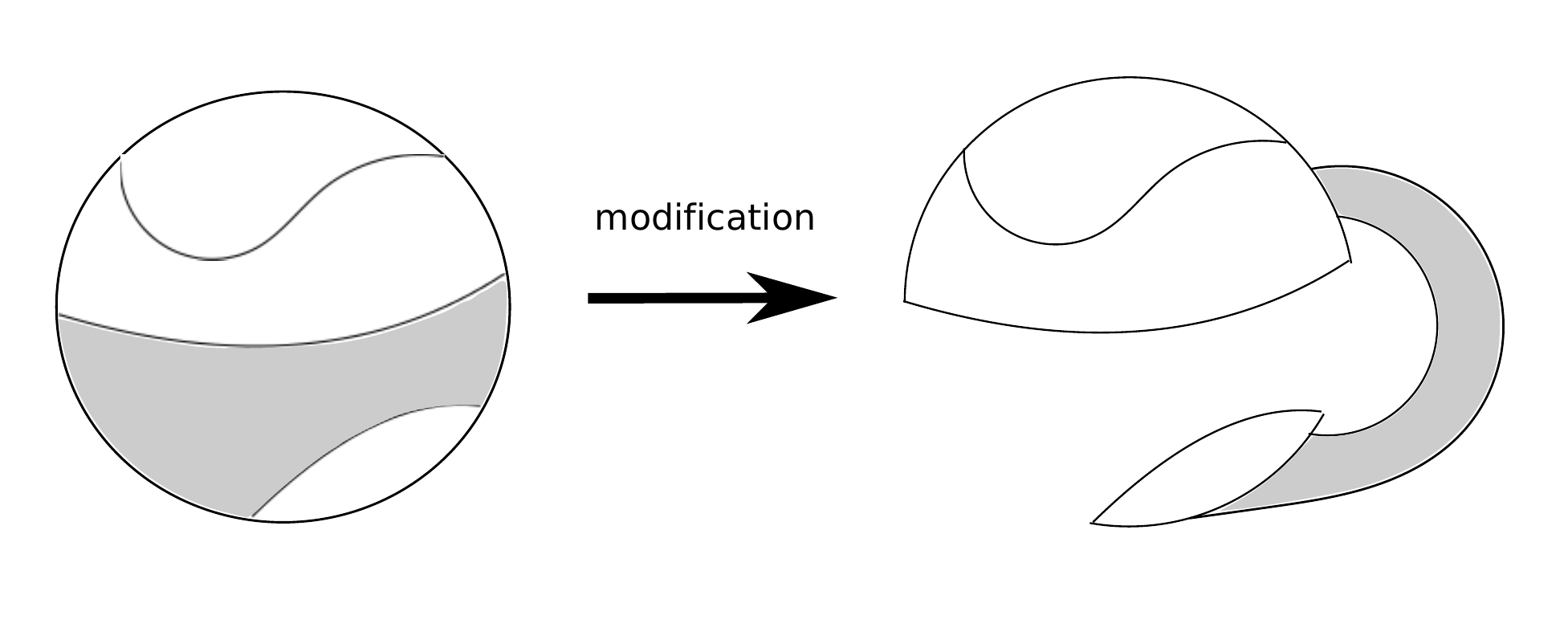}
  \put(-3.8,1.8){$\K$}
  \put(-5,1.3){$\D$}  
  \put(-4.3,0.64){$\text{N}$} 
  \put(-2,1.8){$\widetilde{\K}$} 
  \put(-4.1,0.10){$\B$} 
  \put(-1.2,0.10){$\widetilde{\B}$} 
\caption{Schematic of basic idea of the modification}
\end{figure}
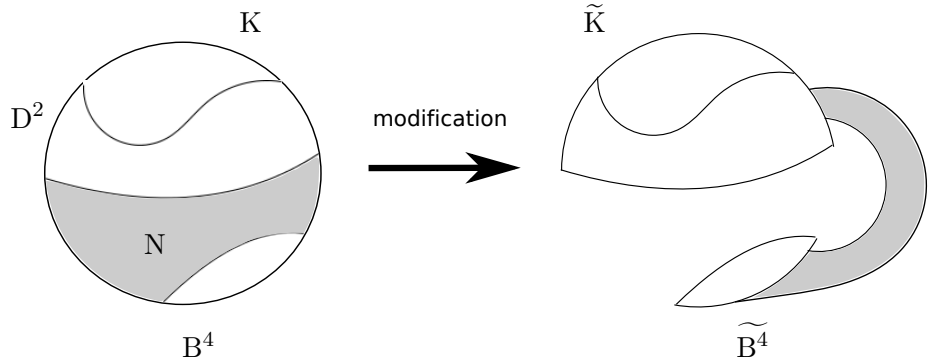

To be more precise about the technique we will need some definitions.

\begin{definition}
Let $\{ \eta_1, \eta_2 \}$ be an oriented link in $\CCC - K$ and let $\phi_\A : \C \times [0, 1] \hookrightarrow \B$ be a smooth proper embedding of an annulus with $\text{Im}(\left. \phi_\A \right |_{\C \times \{ 0 \}})= \eta_1$, $\text{Im}(\left. \phi_\A \right |_{\C \times \{ 1 \}})= \eta_2$, $\text{Im}(\phi_\A)= \A$, and $l=lk(\eta_1,\eta_2)$.
We will say $(\{ \eta_1, \eta_2 \}, \phi_A)$ is \textit{$l$-nice} if it satisfies the following:
\begin{enumerate}
\item $\A \cap \D = \emptyset$
\item The link $\{ \eta_1, \eta_2 \}$ is isotopic to $L_l$ in $\CCC$ where $l$ is an integer and $L_l$ is the two component link described by black curves in Figure $5$.
\item $<[c]^n \cdot [\mu_1]> = <[\mu_1]>$ in $\pi_1(\B - \N)$ where $c$ is the knot disjoint from $L_l$ described by the red curve in Figure $5$ and $\mu_1$ is a meridian of $\eta_1$ as above. 
\end{enumerate}
\end{definition}

\begin{remark}$ $
\begin{enumerate}
\item Condition $(3)$ from Defnition 2.1 is technical condition imposed to make sure resulting manifold after performing an $n$-twist annulus modification is simply connected.
\item When the link $\{ \eta_1, \eta_2 \}$ is isotopic to $L_0$ in $\CCC$, the condition $(3)$ from Definition 2.1 is automatically satisfied. Note that $[c]$ represents trivial element in $\pi_1(\CCC - \text{N}(\{ \eta_1, \eta_2 \}))$, hence it represents trivial element in $\pi_1(\B - \N)$. So we have $<[c]^n \cdot [\mu_1]> = <[id]^n \cdot [\mu_1]> =<[\mu_1]>$ in $\pi_1(\B - \N)$.
\end{enumerate}
\end{remark}

\begin{figure}[h]
\centering
\includegraphics[width=3.1in]{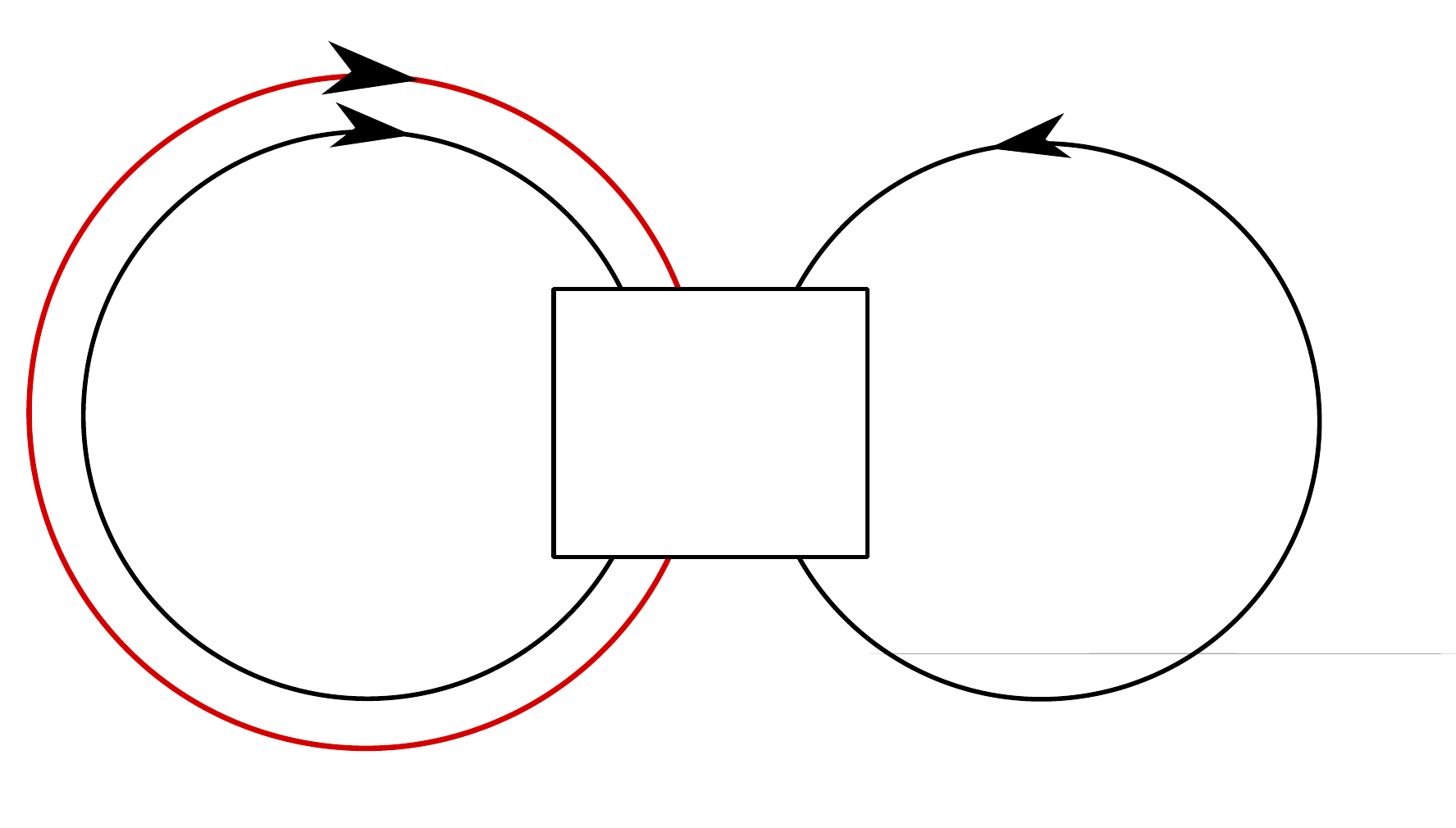}
  \put(-1.55,0.79){$l$}  
   \put(-2.3,0.4){$\eta_1$} 
   \put(-2.3,0.05){$c$}
   \put(-0.9,0.4){$\eta_2$} 
\caption{The link $L_l = \{\eta_1, \eta_2 \}$ where, the box containing $l$ represents the $l$ full twists, and the knot $c$}
\end{figure}
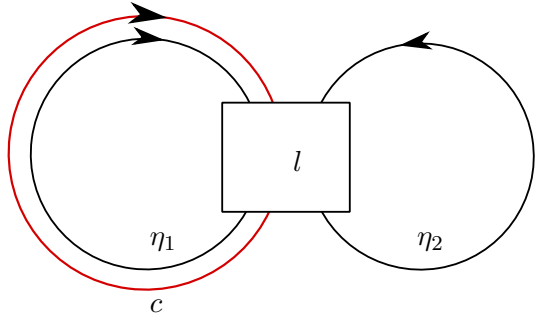

When $(\{ \eta_1, \eta_2 \}, \phi_A)$ is $l$-nice, performing an $n$-twist annulus modification on $\B$ along at $\phi_\A$ gives us the following main theorem.

\begin{theorem}\label{Main}Let $K$ be a smoothly slice knot bounding a smooth disk $\D$ in the standard $\B$, and let $(\{ \eta_1, \eta_2 \}, \phi_A)$ be $l$-nice. Then $\frac{nl+1}{n}$ Dehn surgery on $\eta_1$ followed by $\frac{nl-1}{n}$ Dehn surgery on $\eta_2$ will produce an exotically slice knot $\K_{(\phi_\A,n)} \subseteq \CCC$ for any integer $n$, where $\K_{(\phi_\A,n)}$ is the image of $\K$ in the new $3$-manifold $\CCC$.\end{theorem}

\begin{proof}
We will perform an $n$-twist annulus modification on $\B$ at $\phi_\A$ to get $\B_{(\phi_\A,n)}$.

It is easy to check that $\B_{(\phi_\A,n)}$ is a homology $\B$ by using a Mayer-Vietoris sequence. We omit this detail.

We need to show that $\B_{(\phi_\A,n)}$ is simply connected. We will use Seifert-van Kampen theorem to see $\B_{(\phi_\A,n)}$ is simply connected. First we apply it to $\B$ to get the following equations where $i_1$ is natural inclusion of $\C \times \partial \D \times [0 , 1]$ into $\B - \N$ and $i_2$ is natural inclusion of $\C \times \partial \D \times [0 , 1]$ into $\N$.
\begin{align*}
 \{id\} = \pi_1(\B) &= \frac{\pi_1(\B - \N) * \pi_1(\N)}{<(i_1)_*([\mu_1])=(i_2)_*([\mu_1]),(i_1)_*([\lambda_1])=(i_2)_*([\lambda_1])>}\\
 &= \frac{\pi_1(\B - \N)}{<(i_1)_*([\mu_1])>} \text{ (Since, } \pi_1(\N) = \mathbb{Z} \text{ is generated by } (i_2)_*([\lambda_1]) \text{.)}
\end{align*}

Hence, we have $\pi_1(\B - \N) / <(i_1)_*([\mu_1])> = \{id\}$. Now we apply Seifert-van Kampen theorem to $\B_{(\phi_\A,n)}$. Recall that $f_n$ was a map from $\C \times \partial \D \times [0 , 1]$ to itself from above and let $y=[(i_2)_*([\mu_1])]$ be a generator of of $\pi_1(\N) = \mathbb{Z}$.

\begin{align*}
\pi_1(\B_{(\phi_\A,n)}) &= \frac{\pi_1(\B - \N) * \pi_1(\N)}{<(i_1)_*([\mu_1])=(i_2 \circ {f_n}^{-1})_*([\mu_1]),(i_1)_*([\lambda_1])=(i_2 \circ {f_n}^{-1})_*([\lambda_1])>}\\
  &= \frac{\pi_1(\B - \N)  \text{ }* <y>}{<(i_1)_*([\mu_1])=y^{-n},(i_1)_*([\lambda_1])=y^{nl+1}>}\\
  &= \frac{\pi_1(\B - \N)}{<({(i_1)_*([\mu_1])}^{l} \cdot (i_1)_*([\lambda_1]))^{n} \cdot (i_1)_*([\mu_1])>}\\
  &= \frac{\pi_1(\B - \N)}{<([c])^{n} \cdot (i_1)_*([\mu_1])>} \text{ (Since, } [c] = {(i_1)_*([\mu_1])}^{l} \cdot (i_1)_*([\lambda_1]) \text{ in } \pi_1(\CCC - \text{N}(\{ \eta_1, \eta_2 \})) \text{.)}\\
  &= \frac{\pi_1(\B - \N)}{<(i_1)_*([\mu_1])>}\text{ (By the condition (3) in Definition 2.1.)}\\
  &= \{id\} \text{ (By observation above.)}
\end{align*}
This shows  $\B_{(\phi_\A,n)}$ is simply connected as we needed.

What is now left to do is to understand what happens on the boundary. Notice that $\partial \B_{(\phi_\A,n)}$ is the result of Dehn surgeries on $\eta_1$ and $\eta_2$, since we are simply removing two solid tori from $\partial \B$ and glueing them back differently. Hence it is enough to calculate the coefficient on both curves to specify the boundary. We are using $f_n$ to glue $\C \times \D \times [0, 1]$ to $\B - \N$, so we have the following identifications:

\begin{itemize}
\item $\C \times \{ 1 \} \times \{ 0 \}$ is identified with $\lambda_1+l\mu_1$.
\item $\{ 1 \} \times \partial \D \times \{ 0 \}$ is identified with $n\lambda_1+(nl+1)\mu_1$.
\item $\C \times \{ 1 \} \times \{ 1 \}$ is identified with $\lambda_2 + l\mu_2$.
\item $\{ 1 \} \times \partial \D \times \{ 1 \}$ is identified with $n\lambda_2+(nl-1)\mu_2$.
\end{itemize}

Recall that $\lambda_1$ and $\lambda_2$ are the preferred longitudes of $\eta_1$ and $\eta_2$ respectively, and $\mu_1$ and $\mu_2$ are meridians of $\eta_1$ and $\eta_2$ respectively. So that meridian of $\C \times \D \times \{ 0 \}$ is identified with $n\lambda_1+(nl+1)\mu_1$ and meridian of $\C \times \D \times \{ 1 \}$ is identified with $n\lambda_2+(nl-1)\mu_2$. This shows that $\frac{nl+1}{n}$ is the coefficient for $\eta_1$ and $\frac{nl-1}{n}$ for $\eta_2$ which implies $\partial \B_{(\phi_\A,n)}$ is the top left picture in Figure $6$. By doing Rolfsen twists, it is easy to check $\partial \B_{(\phi_\A,n)}$ is $\CCC$, which is described in Figure $6$.

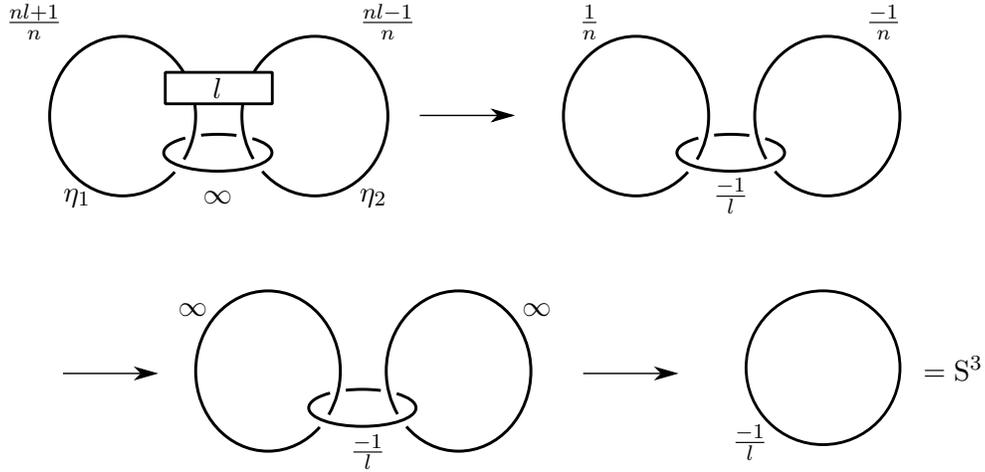
\begin{figure}[h]
\centering
\includegraphics[width=5in]{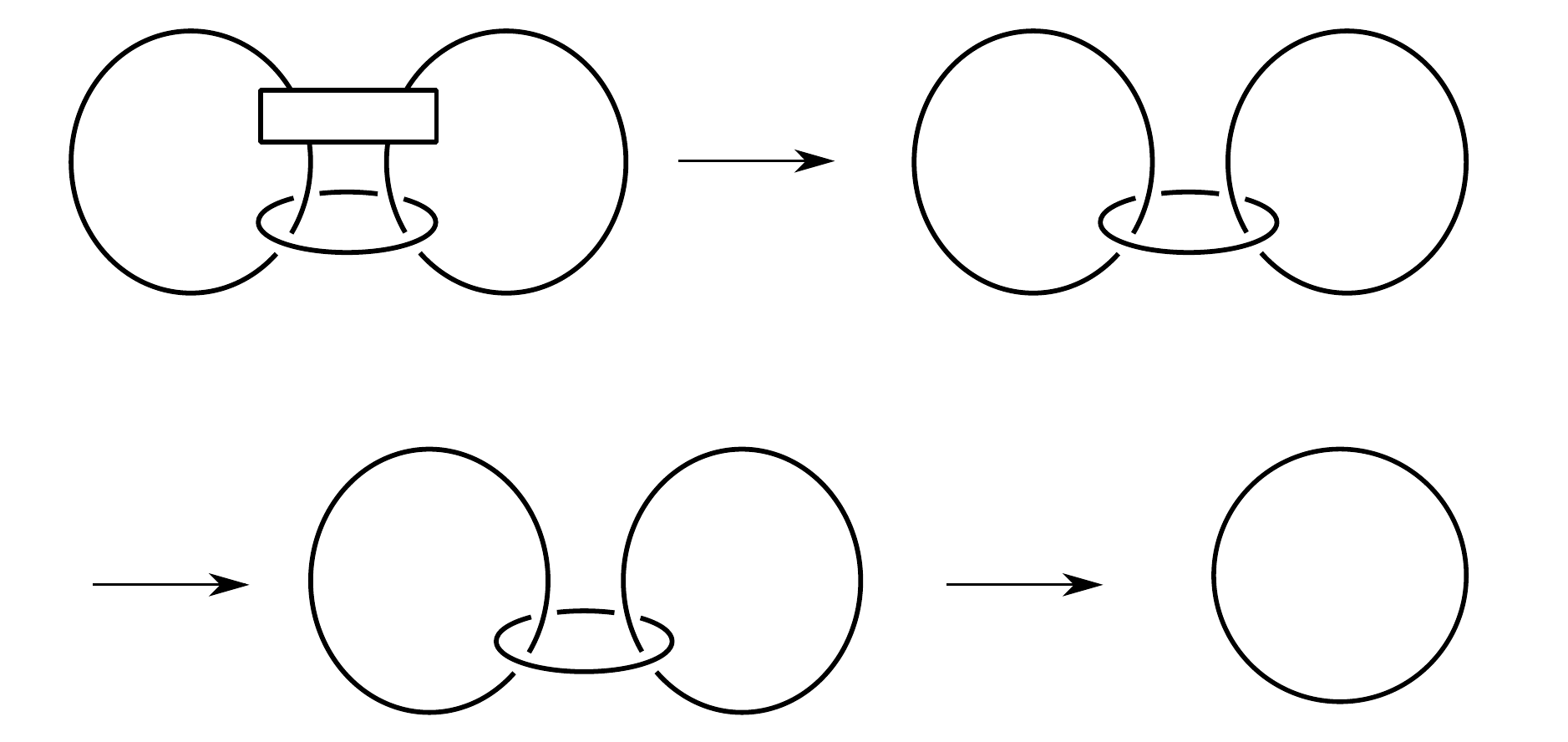}
  \put(-5,2.3){$\frac{nl+1}{n}$}
  \put(-3.2,2.3){ $\frac{nl-1}{n}$}  
    \put(-4.7,1.4){$\eta_1$}
  \put(-3.2,1.4){ $\eta_2$}  
  \put(-3.97,1.94){ $l$}  
  \put(-3.97,1.39){$\infty$} 
  \put(-2,2.3){$\frac{1}{n}$} 
  \put(-0.5,2.3){$\frac{-1}{n}$} 
  \put(-1.3,1.39){$\frac{-1}{l}$}
  \put(-4.1,0.8){$\infty$} 
  \put(-2.3,0.8){$\infty$} 
  \put(-3.2,0.05){$\frac{-1}{l}$} 
  \put(-1.2,0.10){$\frac{-1}{l}$} 
    \put(-0.2,0.46){$=\CCC$} 
\caption{The first arrow is given by a $-l$ Rolfsen twist on the $\infty$ circle. The second arrow is given by a $-n$ Rolfsen twist on $\eta_1$ and $n$ a Rolfsen twist on $\eta_2$. The third arrow is given by deleting a component with coefficient $\infty$.}
\end{figure}

Thus by the 4-dimensional topological Poincar\'e conjecture we can conclude that $\B_{(\phi_\A,n)}$ is homeomorphic to $\B$ \cite[Theorem $1.6$]{Fr84}, which implies that $\K_{(\phi_\A,n)}$ is exotically slice for any integer $n$.\end{proof}

\section{Special Case}\label{special case}
In this section we will discuss a special case of Section $2$, which guarantees that the resulting manifold $\B_{(\phi_\A,n)}$ is diffeomorphic to $\B$.

\begin{figure}[h]
\centering
\includegraphics[width=2in]{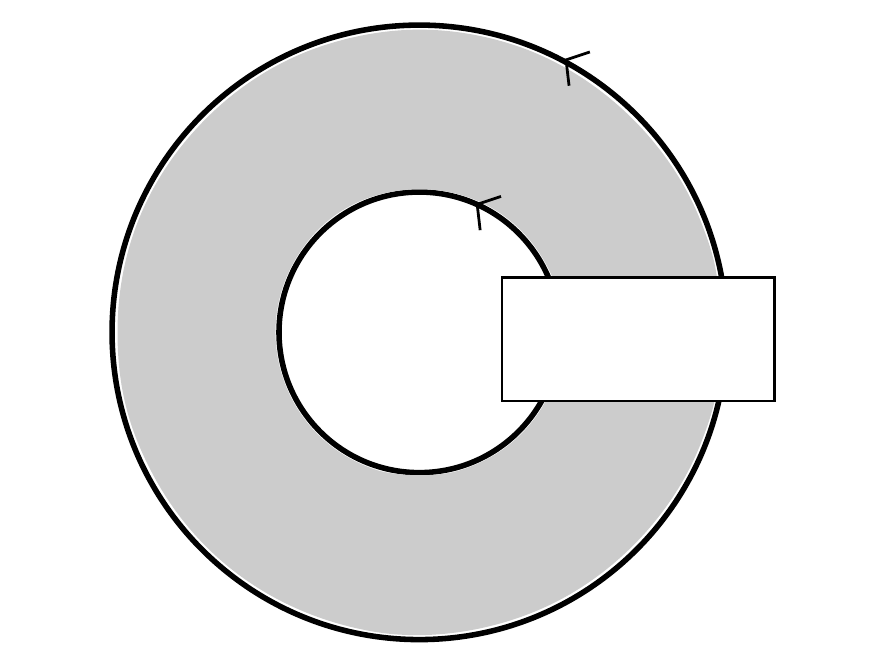}
  \put(-0.6,0.7){$l$}
    \put(-1.1,0.85){$\eta_1$}
      \put(-0.5,1.3){$\eta_2$}  
\caption{The annulus $\A_l$}
\end{figure}

\begin{definition}
Let $\phi_{\A_l} : \C \times [0, 1] \hookrightarrow \B$ be a smooth proper embedding of an annulus with $\text{Im}(\phi_{\A_{l}})= \A_{l}$ where $\A_l$ is obtained by pushing in the interior of the annulus, described in Figure $7$. 
We will call $(\{ \eta_1, \eta_2 \}, \phi_A)$ \textit{$l$-standard} if $(\{ \eta_1, \eta_2 \}, \phi_A)$ is $l$-nice and, further, if the annulus $\A$ that is bounded by $\eta_1$ and $-\eta_2$ is smoothly isotopic through proper embeddings to $\A_l$. 
\end{definition}

\begin{remark} When the link $\{ \eta_1, \eta_2 \}$ is isotopic to link $L_l$ and if it bounds an annulus $\A$, smoothly isotopic through proper embeddings to $\A_l$, then the condition $(3)$ from Definition 2.1 is automatically satisfied. Note that the curve $c$ bounds a smoothly embedded disk in $\B-\N$ which is described in Figure $8$. Hence $[c]$ represents a trivial element in $\pi_1(\B-\N)$ and we see that the condition is satisfied.
\end{remark}

\begin{figure}[h]
\centering
\includegraphics[width=5.5in]{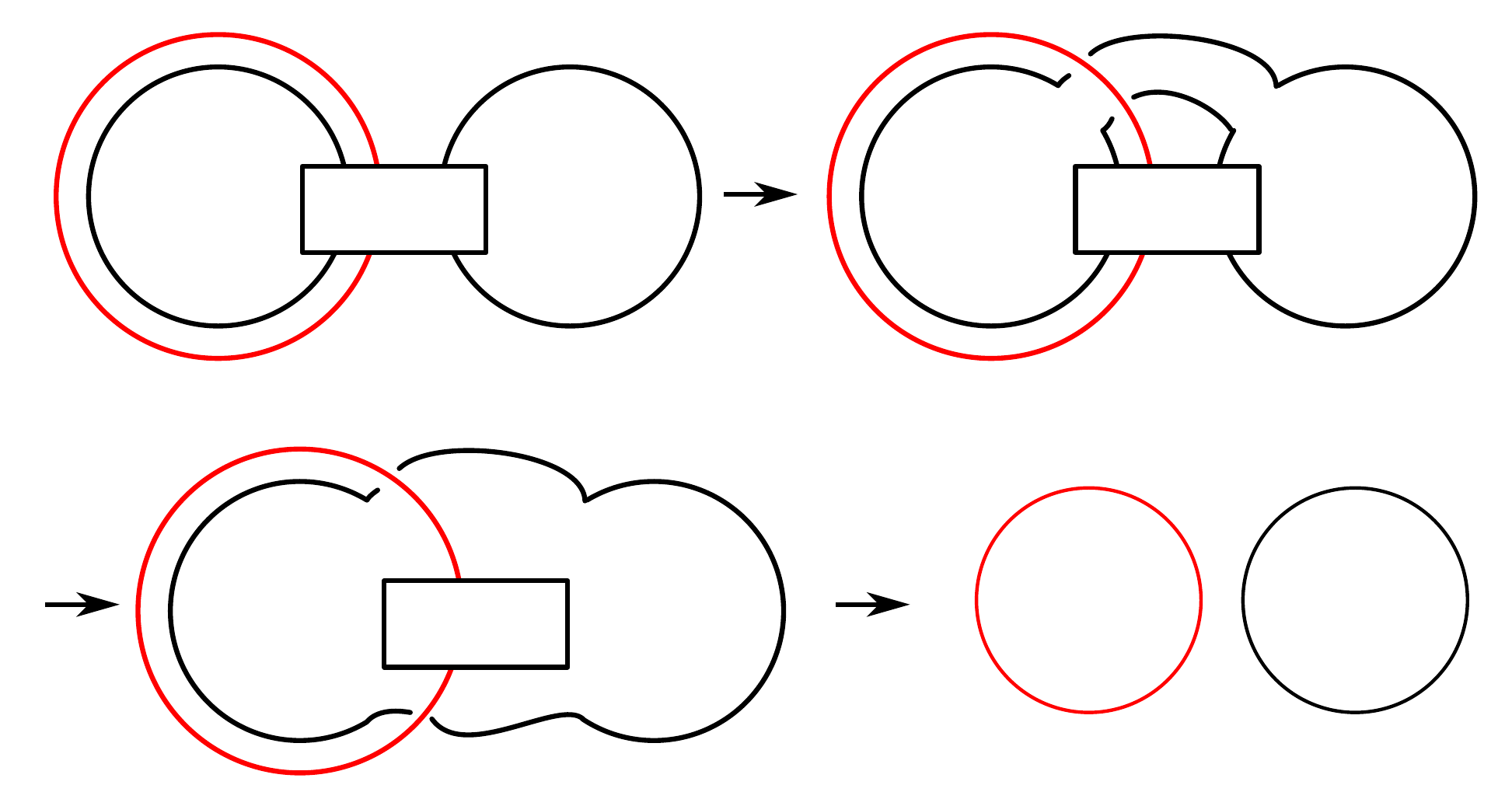}
  \put(-4.07,2.1){$l$}  
   \put(-1.25,2.1){$l$}  
      \put(-3.8,0.6){$l$}
   \put(-4.7,1.9){$\eta_1$} 
   \put(-3.5,1.9){$\eta_2$} 
   \put(-4.7,1.45){$c$}
   \put(-1.85,1.45){$c$}
      \put(-4.4,-0.1){$c$}
   \put(-1.55,0.1){$c$}
\caption{Top left figure is the link $\{\eta_1, \eta_2\}$ and the knot $c$. The second picture is obtained by performing band sum between $\eta_1$ and $\eta_2$. The third and fourth pictures are obtained by isotopy of the black curve. Note that the knot $c$ becomes completely disjoint from the annulus $\A$ after the band sum between $\eta_1$ and $\eta_2$.}
\end{figure}
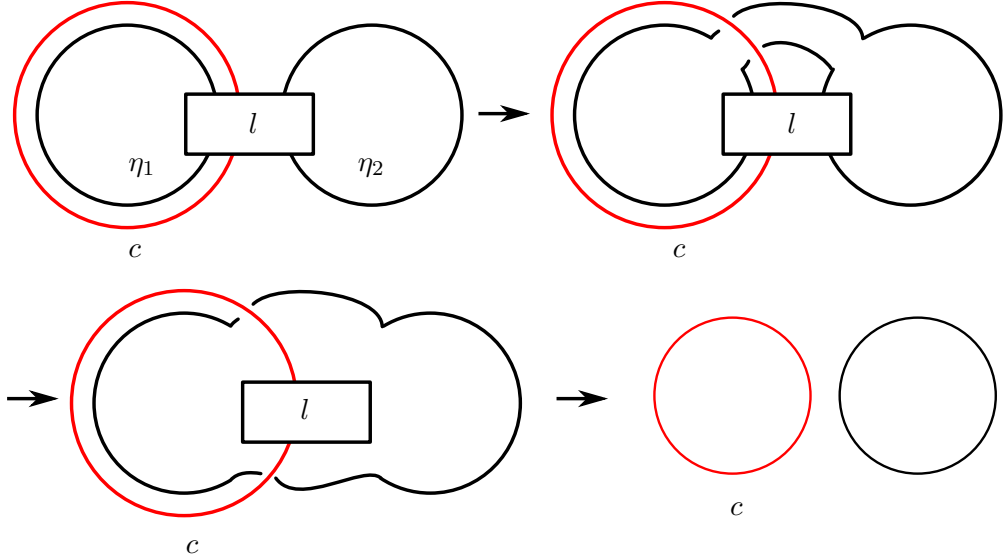

Then we have the following theorem. Note that this is an analogue of Theorem $3.1$ in \cite{CD14}.

\begin{theorem}\label{Diffeo}Let $\K$ be a smoothly slice knot bounding a smooth disk $\D$ in the standard $\B$, and let $(\{ \eta_1, \eta_2 \}, \phi_A)$ be $l$-standard. Then $\frac{nl+1}{n}$ Dehn surgery on $\eta_1$ followed by $\frac{nl-1}{n}$ Dehn surgery on $\eta_2$ will produce an smoothly slice knot $\K_{(\phi_\A,n)}\subseteq \CCC$ for any integer $n$, where $\K_{(\phi_\A,n)}$ is the image of $\K$ in the new $3$-manifold $\CCC$.\end{theorem}

\begin{proof}
By Theorem $2$, the only thing that we need to show is that $\B_{(\phi_\A,n)}$ is diffeomorphic to the standard $\B$ and not just homeomorphic for any integer $n$, if $(\{ \eta_1, \eta_2 \}, \phi_A)$ is $l$-standard.

Note that if $\phi_\A$ and $\phi_{\A'}$ are smooth proper embedding of annulus into $\B$ that are smoothly isotopic through proper embeddings, then $\B_{(\phi_\A,n)}$ is diffeomorphic to $\B_{(\phi_{\A'},n)}$. This can be easily checked by the ambient isotopy theorem \cite[Chapter $8$, Theorem $1.3$]{Hir94}.

By using this we will first show $\B_{(\phi_\A,n)}$ is diffeomorphic to the standard $\B$ when $(\{ \eta_1, \eta_2 \}, \phi_A)$ is $0$-standard. We can think of $\B$ as $\text{B}^3 \times [0 , 1]$, so we have a natural smooth proper embedding of an annulus $\text{U} \times [0 , 1] \subseteq \text{B}^3 \times [0 , 1] = \B$, where $\text{U}$ is the unknot. Then observe $\text{U} \times [0 , 1] \subseteq \B$ is isotopic to $\A_0$; one could visualize this by pulling boundary of $\text{U} \times [0 , 1]$ to $\partial \text{B}^3 \times [0 , 1]$. Hence we can conclude that $\B_{(\phi_{\A_0},n)}$ is diffeomorphic to $\text{B}^3_{\frac{1}{n}} \times [0 , 1] = \text{B}^3 \times [0 , 1] = \B$, where $\text{B}^3_{\frac{1}{n}}$ is ${\frac{1}{n}}$ Dehn surgery along the unknot. Thus, $\B_{(\phi_\A,n)}$ is diffeomorphic to standard $\B$ for any integer $n$ when $(\{ \eta_1, \eta_2 \}, \phi_A)$ is $0$-standard.

For $l\neq 0$, we need to define one more modification. Let $\M$ be a compact $4$-manifold with non-empty boundary, and assume $\M$ is a integer homology four ball. Let $\phi_D : \D \hookrightarrow \M$ be a smooth proper embedding of a disk with $\text{Im}(\phi_D)= D$. Then carve out a tubular neighborhood of $D$ and attach a $2$-handle along the meridian of $\partial D$ with framing $l$. We will call this a \textit{$l$-disk modification on $\M$ at $\phi_D$} and denote the resulting manifold as $D_l(M)$. 

Let $D_0$ be a disk which is obtained by pushing in the interior of the smoothly embedded disk in $\CCC$ to the standard $\B$. Notice that if $D$ is a proper embedding of a disk in $\B$ and if it is isotopic through proper embedding to $D_0$, then a $l$-disk modification on $\B$ is simply adding a canceling $1$-handle / $2$-handle pair which does not change the $4$-manifold.

We will fix two particular disjoint proper smooth embeddings $\phi_{\widetilde{A}}$ and $\phi_{\widetilde{D}}$ which are described in Figure $9$. We will denote $\text{Im}(\phi_{\widetilde{\A}})= \widetilde{\A}$ and $\text{Im}(\phi_{\widetilde{D}})= \widetilde{D}$. Notice $\widetilde{\A}=U \times [0,1] \subseteq \text{B}^3 \times [0 , 1]$ where $U$ is the unknot, and $\widetilde{D}$ is $\text{Arc} \times [0,1] \subseteq \text{B}^3 \times [0 , 1]$. We will do two modifications on $\B$: first an $n$-twist annulus modification along $\phi_{\widetilde{\A}}$ and second an $l$-disk modification at $\phi_{\widetilde{D}}$. Notice that the order of these modifications does not matter so we have the commutative diagram shown in Figure $10$, we have $\widetilde{D}_l(\B_{(\phi_{\widetilde{\A}},n)})=\widetilde{D}_l(\B)_{(\phi_{\widetilde{\A}},n)}$.

\begin{figure}[h]
\centering
\includegraphics[width=4in]{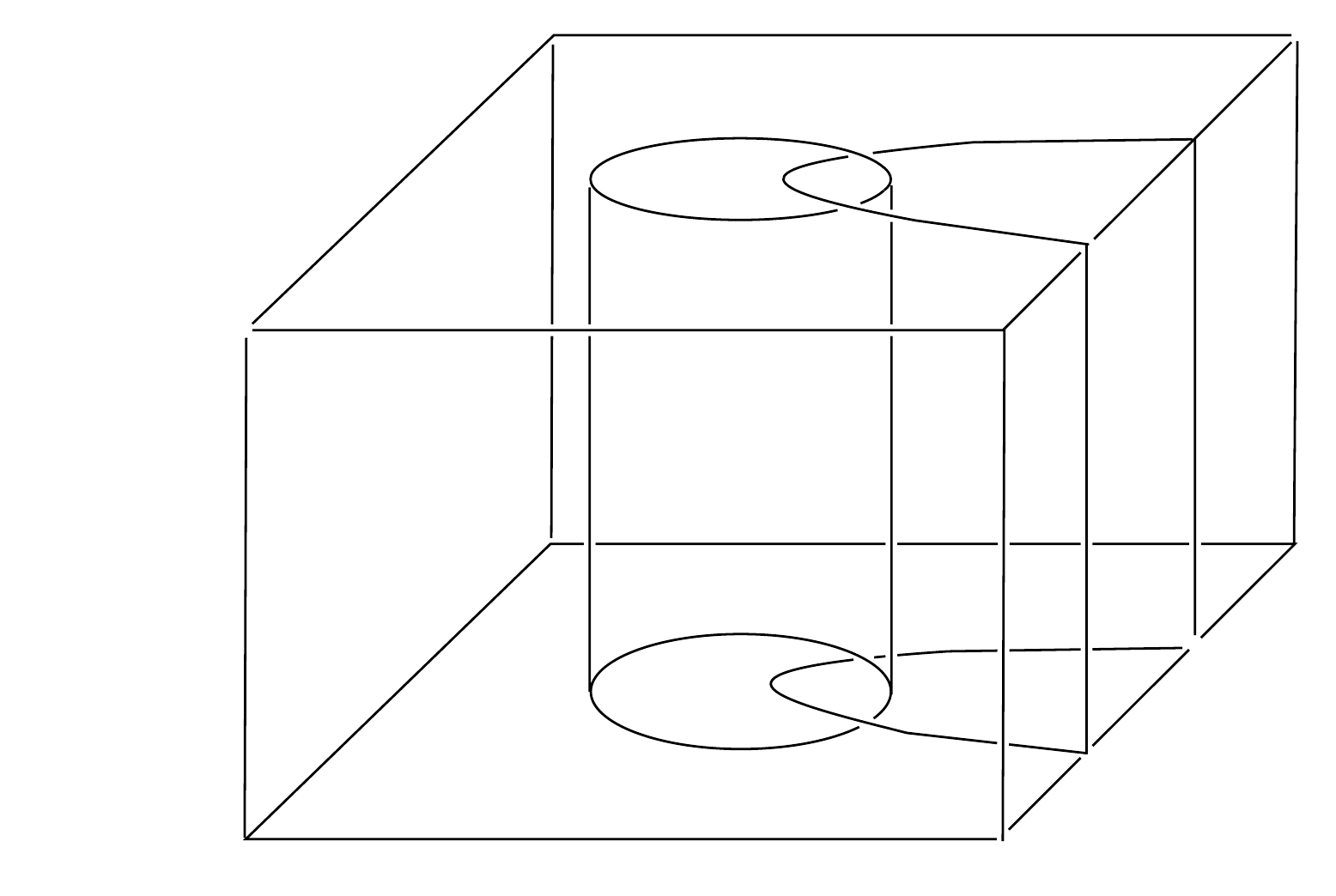}
  \put(-4.6,1){ $\B = \text{B}^3 \times [0 , 1] =$}
  \put(-2,2.4){$U$}  
  \put(-0.8,2.4){Arc}  
    \put(-2.2,0.3){$\widetilde{\A}$}  
  \put(-0.6,0.3){$\widetilde{D}$}  
\caption{Annulus $\widetilde{\A}$ and Disk $\widetilde{D}$}
\end{figure}

\begin{figure}[h]
\centering
\includegraphics[width=3in]{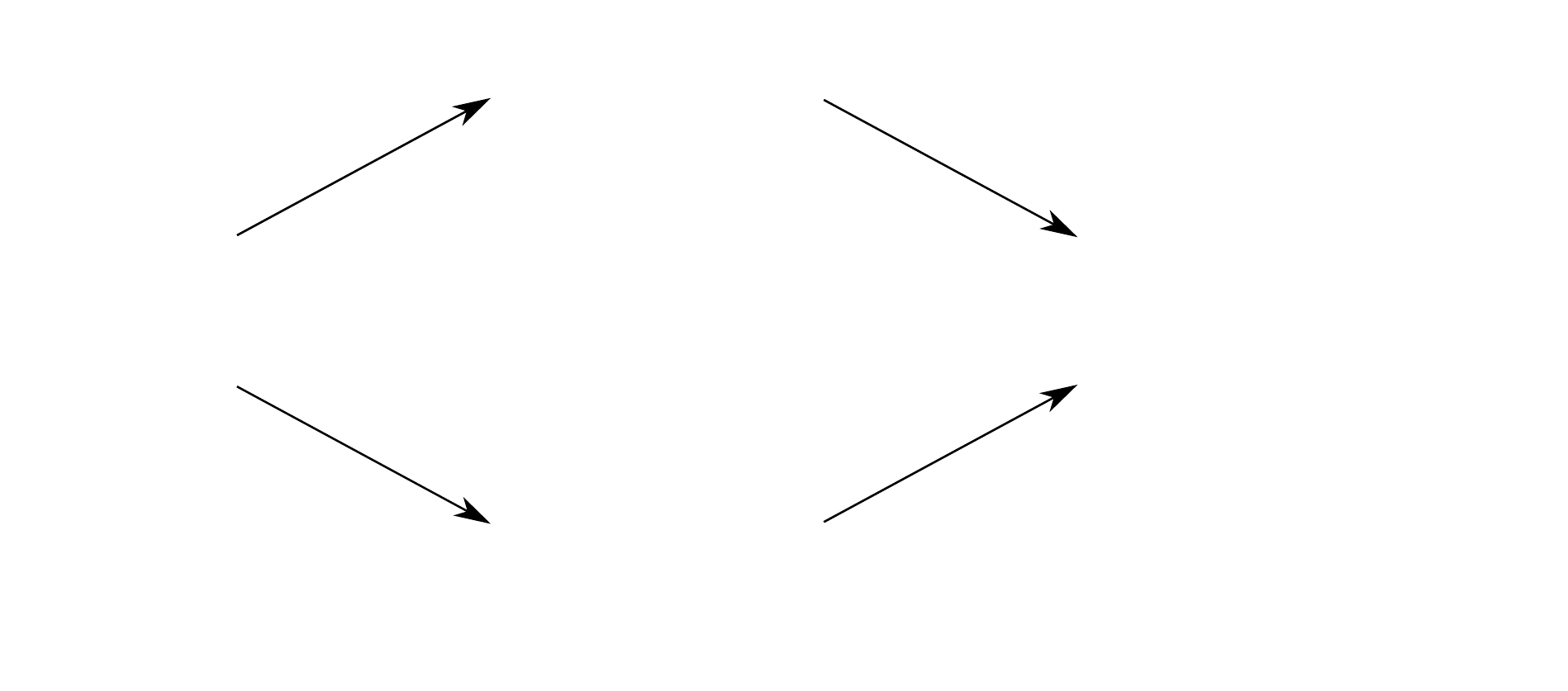}
  \put(-2.9,0.68){$\B$}  
  \put(-1.96,1.2){$\B_{(\phi_{\widetilde{\A}},n)}$}  
  \put(-1.96,0.2){$\widetilde{D}_l(\B)$}  
   \put(-0.9,0.68){$\widetilde{D}_l(\B_{(\phi_{\widetilde{\A}},n)})=\widetilde{D}_l(\B)_{(\phi_{\widetilde{\A}},n)}$}  
     \put(-2.5,1.1){$(1)$} 
       \put(-1.2,1.1){$(2)$} 
         \put(-2.5,0.3){$(3)$} 
           \put(-1.2,0.3){$(4)$} 
\caption{Two modifications on $\B$}
\end{figure}
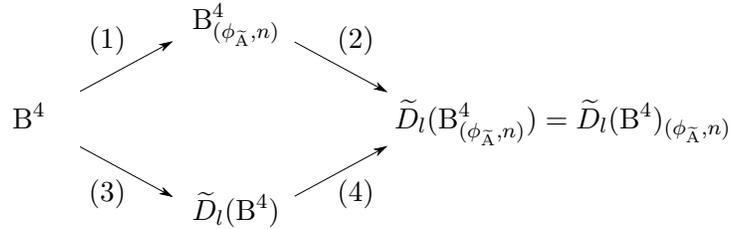

\begin{enumerate}
\item For the map (1), notice that $\phi_{\widetilde{\A}}$ is isotopic to $\A_0$. In that case, we have shown already that $\B_{(\phi_{\widetilde{\A}},n)}$ is diffeomorphic to the standard $\B$.
\item For the map (2), Dehn surgery at each level does not change the $\text{Arc}$. Thus $\widetilde{D} \subseteq \B_{(\phi_{\widetilde{\A}},n)}=\B$ and $D_0\subseteq \B$ are smoothly isotopic through proper embeddings, which implies that $\widetilde{D}_l(\B_{(\phi_{\widetilde{\A}},n)})$ is diffeomorphic to the standard $\B$. 
\item For the map (3), $\widetilde{D} \subseteq \B$ and $D_0\subseteq \B$ are smoothly isotopic through proper embeddings, which implies that $\widetilde{D}_l(\B)$ is diffeomorphic to the standard $\B$. 
\item For the map (4), before performing any modifications we can isotope $\widetilde{A}$ to $\partial \B = \CCC$ away from $\widetilde{D}$. We can visualize this (see Figure $9$) by pushing $\widetilde{A}$ in to $\partial \text{B}^3 \times [0 , 1]$ to the right. Then after the modification, $\widetilde{A} \subseteq \widetilde{D}_l(\B)$ is smoothly isotopic through proper embeddings to an annulus in $\CCC$ which is described in Figure $11$. Hence $\widetilde{A} \subseteq \widetilde{D}_l(\B)=\B$ is smoothly isotopic through proper embeddings to $\A_l \subseteq \B$ which was described in the beginning of the section. Then we can conclude that $\widetilde{D}_l(\B)_{(\phi_{\widetilde{\A}},n)}$ is diffeomorphic to $\B_{(\phi_{\widetilde{\A}_l},n)}$. 
\end{enumerate}
By $(1)$ and $(2)$, we see that $\widetilde{D}_l(\B_{(\phi_{\widetilde{\A}},n)})$ is diffeomorphic to the standard $\B$, and from $(3)$ and $(4)$, $\widetilde{D}_l(\B)_{(\phi_{\widetilde{\A}},n)}$ is diffeomorphic to $\B_{(\phi_{\widetilde{\A}_l},n)}$. Hence $\B_{(\phi_{\widetilde{\A}_l},n)}$ is diffeomorphic to the standard $\B$ for all integers $n$ which concludes the proof.\end{proof}

\begin{figure}[h]
\centering
\includegraphics[width=5in]{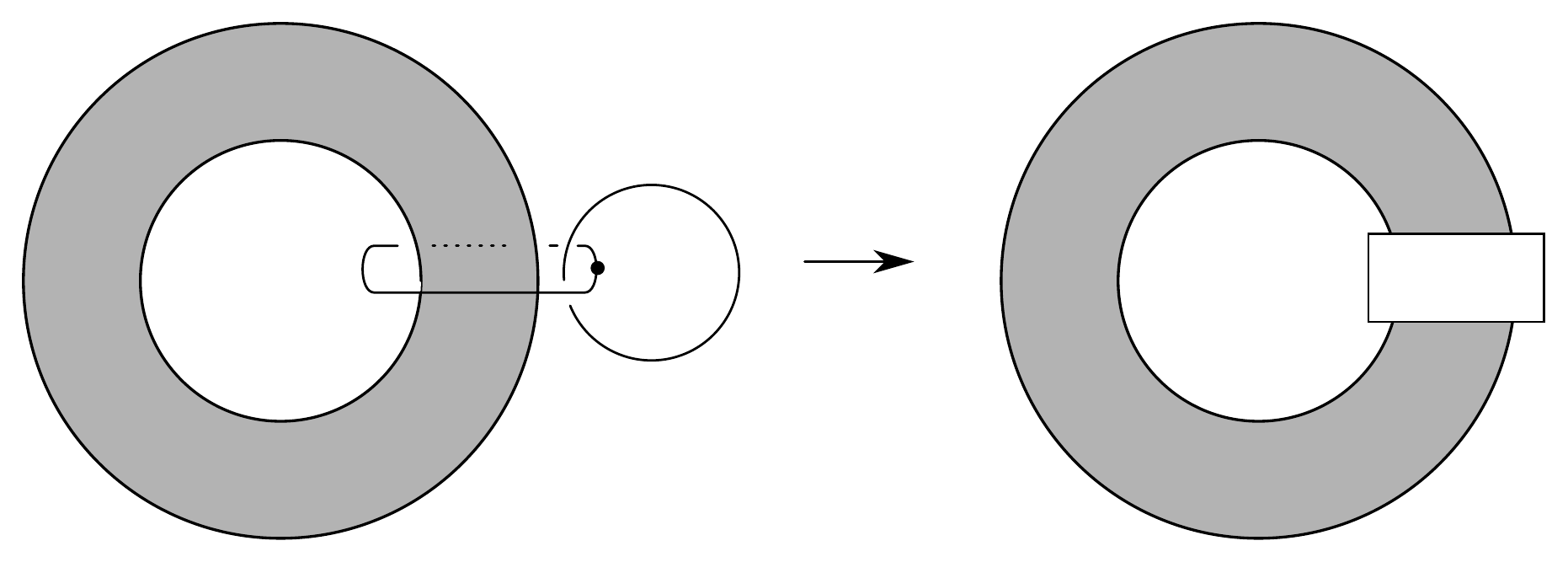}
  \put(-2.9,1.3){$l$}  
  \put(-0.38,0.86){$l$}  
    \put(-2.4,1.1){$(3)$}  
\caption{$\widetilde{\A}$ after $(3)$}
\end{figure}

We end this section by using a result of Scharlemann and a result of Livingston to find a sufficient criterion for $0$-nice $(\{ \eta_1, \eta_2 \}, \phi_A)$ to be $0$-standard.

\begin{theorem} \cite[Main Theorem]{Sc85} Suppose that $\eta_1$ and $\eta_2$ are knots in $\CCC$ which form a split link and that a certain band sum of $\eta_1$ and $\eta_2$ yields the unknot. Then $\eta_1$ and $\eta_2$ are each unknotted and the band sum is connected sum.
\end{theorem}

\begin{theorem} \cite[Theorem $4.2$]{Liv82} Let $F_1$ and $F_2$ be orientable surfaces embedded in $\CCC$, bounding the unlink. After pushing interior of $F_1$ and $F_2$ to $\B$, they are isotopic through proper embeddings if and only if $F_1$ and $F_2$ are homeomorphic.
\end{theorem}

Let $h : \B = \{ (x,y,z,w) \in \mathbb{R}^4 : x^2+y^2+z^2+w^2 \leq1\} \rightarrow \mathbb{R}$ where $h(x,y,z,w)=x^2+y^2+z^2+w^2$ and let $\phi_\A$ be a proper smooth embedding of an annulus in $\B$. By abuse of notation we will refer to critical points of $h \circ \phi_\A$ as critical points of $\phi_\A$. Then we have the following corollary.

\begin{corollary}Let $\K$ be a smoothly slice knot bounding a smooth disk $\D$, and let $(\{ \eta_1, \eta_2 \}, \phi_A)$ be $0$-nice. If $\phi_A$ has one critical point of index one and one critical point of index two then $(\{ \eta_1, \eta_2 \}, \phi_A)$ is $0$-standard. Hence, $\frac{1}{n}$ Dehn surgery on $\eta_1$ and $\frac{-1}{n}$ Dehn surgery on $\eta_2$ will produce a smoothly slice knot $\K_{(\phi_\A,n)} \subseteq \CCC$ for any integer $n$, where $\K_{(\phi_\A,n)}$ is the image of $\K$ in the new $3$-manifold $\CCC$.

\end{corollary} 

\begin{proof} By Theorem $3.1$ it is enough to show that $(\{ \eta_1, \eta_2 \}, \phi_A)$ is $0$-standard. In other words it would be enough to show $\phi_A$ and $\phi_{A_0}$ are smoothly isotopic through proper embeddings, when $\phi_A$ has one critical point of index one and one critical point of index two.

Since $\eta_1$ and $\eta_2$ form a two component unlink, they bound smoothly embedded disks $D_1$ and $D_2$ respectively. A critical point of index one corresponds to a band sum between $\eta_1$ and $\eta_2$ which can be isotoped into $\CCC$; we will call this band $B$. Let $\eta_0$ be the resulting knot after doing the band sum. A critical point of index two corresponds to a disk bounded by $\eta_0$ which also could be isotoped into $\CCC$ hence $\eta_0$ is the unknot. We will call this disk $D_0$.

By Theorem $3.3$ \cite[Main Theorem]{Sc85}, $B$ is connected sum, and hence $B$ does not intersect $D_1$ and $D_2$. Thus we have two disks $D_0$ and $D_1\cup B\cup D_2$ in $\CCC$ bounded by $\eta_0$. Since any two disks bounded by same curve in $\CCC$ can be isotoped into each other, we can isotope $D_0$ into $D_1\cup B\cup D_2$ and then push it slightly off $D_1\cup B\cup D_2$, so that they are disjoint. This gives you an annulus that is cobounded by $\eta_1$ and $\eta_2$, namely $D_0 \cup B$.

Thus we have isotoped $\phi_A$ into $\CCC$. By Theorem $3.4$ \cite[Theorem $4.2$]{Liv82} there is only one isotopy class of embedding of an annulus into $\CCC$ which sends the boundary components to the unlink. We see that $\phi_A$ and $\phi_{A_0}$ are smoothly isotopic through proper embeddings. Then we can apply Theorem $3.1$ to conclude our proof.
\end{proof}


\section{Application 1 : An example of a slice knot with non-slice derivatives}\label{Application 1}
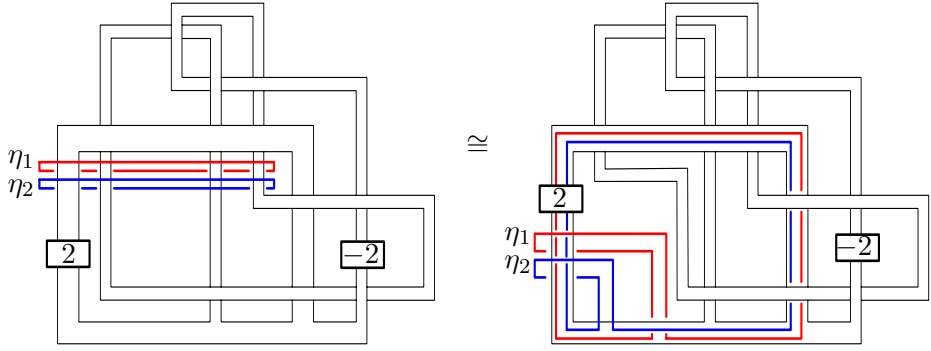
\begin{figure}[h]
\centering
\includegraphics[width=5.1in]{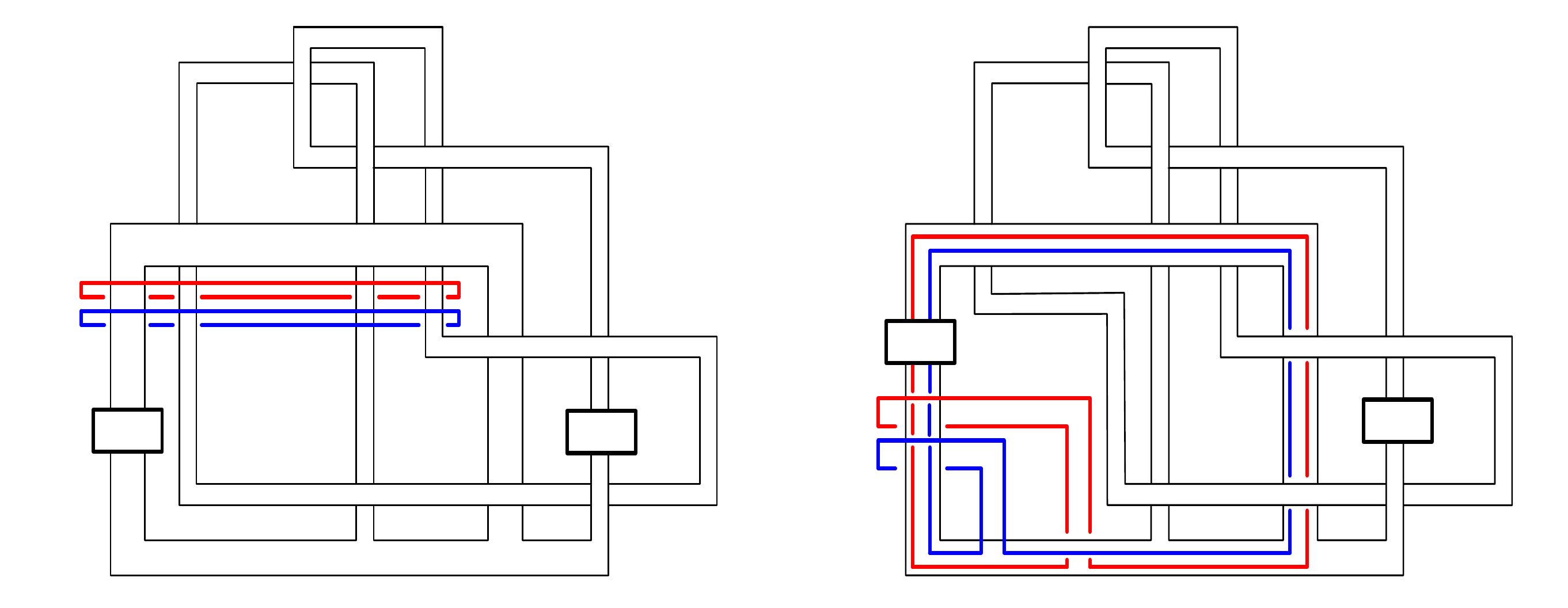}
  \put(-5,1.05){$\eta_1$}  
    \put(-5,0.9){$\eta_2$} 
     \put(-4.72,0.53){$2$}  
      \put(-3.25,0.53){$-2$}  
      \put(-2.4,0.65){$\eta_1$}  
    \put(-2.4,0.5){$\eta_2$} 
           \put(-2.15,0.82){$2$}  
  \put(-0.67,0.57){$-2$}  
    \put(-2.6,1.1){$\cong$}
\caption{The knot $\R$, the two component unlink $\eta_1$, and $\eta_2$ ($\cong$ means isotopic)}
\end{figure}

In this section we will fix a smoothly slice knot $\R$, and a two component unlink $\eta_1$, and $\eta_2$ as in Figure $12$. Note that the core of a second band is a derivative of $\R$ and in fact it is the Stevedore's knot, which implies that $\R$ is smoothly slice. The last picture in Figure $13$ is the $(2,0)$-cable of Stevedore's knot which is concordant to $(2,0)$-cable of unknot which means it is a slice link. Hence we can cap off a slice link with three disks in the last picture in Figure $13$. Then black curves in Figure $13$ describe slice disk $\D$ for $\R$ and red curves in Figure $13$ describe $\phi_\A$ such that $\text{Im}(\phi_\A)= \A$ bounds $\eta_1$ and $-\eta_2$. Now, it is easy to see that $(\{\eta_1, \eta_2\},\phi_\A)$ is $0$-nice, since there was no intersection between black curves and red curves in Figure $13$. Further, $\phi_\A$ has one critical point of index one and one critical point of index two since there was only one band sum between $\eta_1$ and $\eta_2$, so we can use Corollary $3.2$ to conclude $\R_1$, which is obtained by $1$ surgery on $\eta_1$ and $-1$ surgery on $\eta_2$, is a smoothly slice knot. Now we show that the knot $\R_1$ is an example of a slice knot with non-slice derivatives. Note that this is an analogue of Proposition $5.2$ in \cite{CD14}.

\begin{figure}[h]
\centering
\includegraphics[width=5.15in]{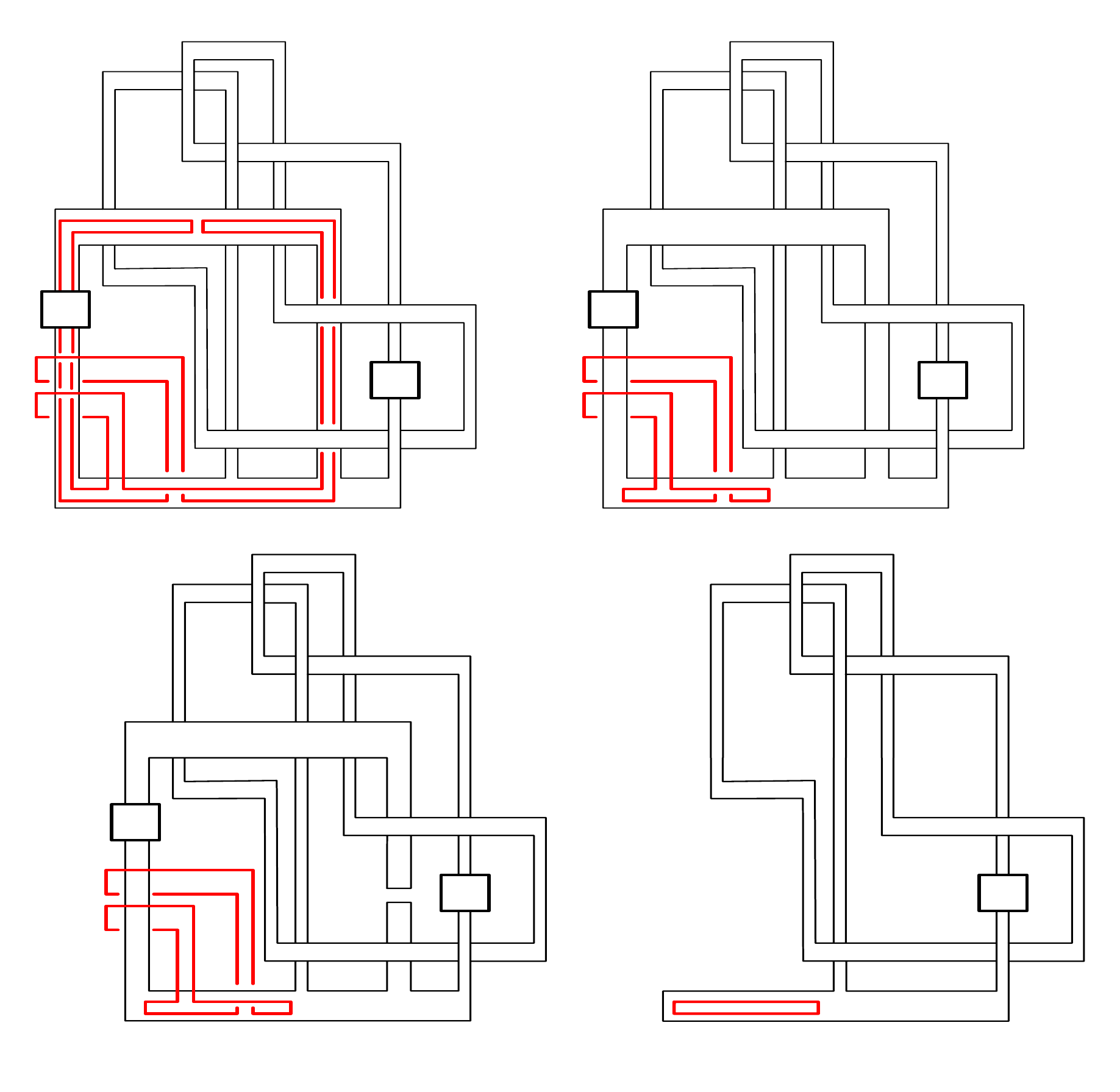}
  \put(-4.9,3.55){$2$}  
  \put(-2.36,3.55){$2$}  
  \put(-3.44,3.22){$-2$}  
  \put(-0.91,3.22){$-2$}  
  \put(-4.56,1.18){$2$} 
  \put(-3.11,0.85){$-2$}  
  \put(-0.63,0.85){$-2$}  
    \put(-5.26,1.07){$\Rightarrow$}
        \put(-2.76,3.37){$\cong$}
                \put(-2.36,1.07){$\cong$}
\caption{The first figure is obtained by adding a band between $\eta_1$ and $-\eta_2$. The third figure is obtained by performing a ribbon move on $\R$ ($\cong$ means isotopic)}
\end{figure}
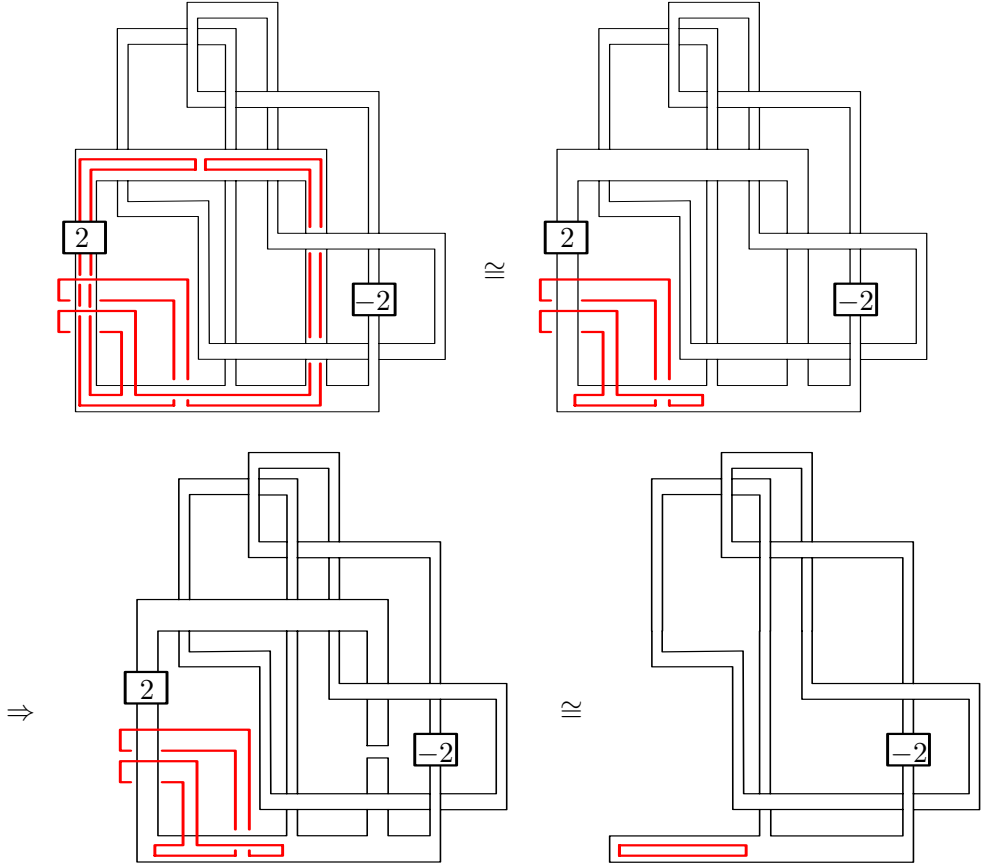

\begin{theorem}\label{Main}Let $\R_1$ be a knot described as above. Then $\R_1$ is a smoothly slice knot with non-slice derivatives.\end{theorem}

\begin{proof}
By Corollary $3.2$, $\R_1$ is a smoothly slice, so it is enough to show that $\R_1$ has non-slice derivatives.

Let $\F$ be the Seifert surface of $\R$ described in Figure $14$. Let $x_1$ and $x_2$ be the cores of the bands of $\F$. Then $\{[x_1],[x_2]\}$ is a basis for $H_1(\F)$ and the Seifert matrix with respect to $\{[x_1],[x_2]\}$ is
$M=\begin{pmatrix}
  2 & 1 \\
  0 & 0
 \end{pmatrix}$, where $M=(m_{i,j})=lx(x_i,{x_j}^+)$ and ${x_j}^+$ is push off of ${x_j}$ in positive direction.
 
 \begin{figure}[h]
\centering
\includegraphics[width=6in]{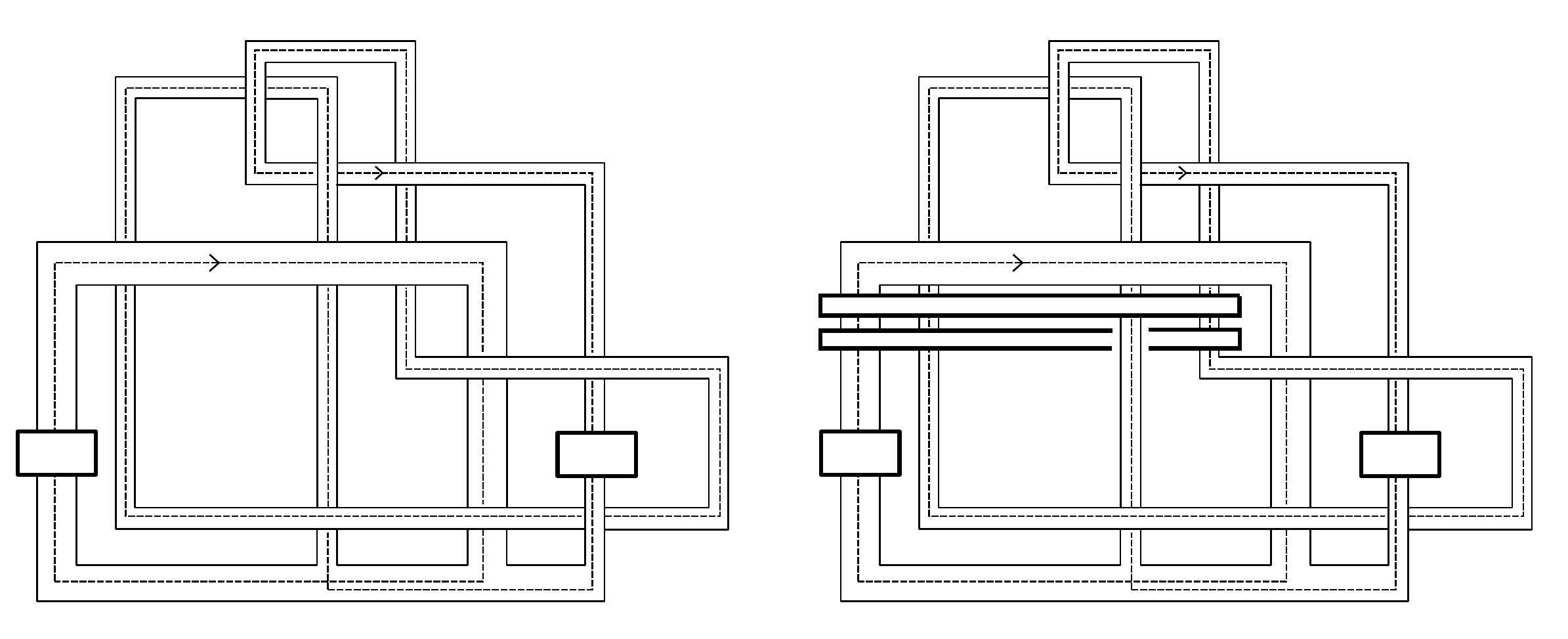}
  \put(-1.6,-0.1){$\widetilde{\F}$} 
    \put(-4.6,-0.1){$\F$} 
           \put(-2.76,0.66){$2$}  
  \put(-0.76,0.66){$-2$}  
             \put(-5.82,0.66){$2$}  
  \put(-3.82,0.66){$-2$}  
    \put(-2.2,1.6){$\widetilde{x_1}$}  
      \put(-1.2,2){$\widetilde{x_2}$}  
          \put(-5.4,1.6){$x_1$}  
      \put(-4.2,2){$x_2$}  
            \put(-2.18,1.245){\scalebox{0.5}{$-1$}} 
                        \put(-2.12,1.12){\scalebox{0.5}{$1$}} 
\caption{Seifert surface $\F$ for $\R$ and Seifert surface $\widetilde{\F}$ for $\R_1$}
\end{figure}
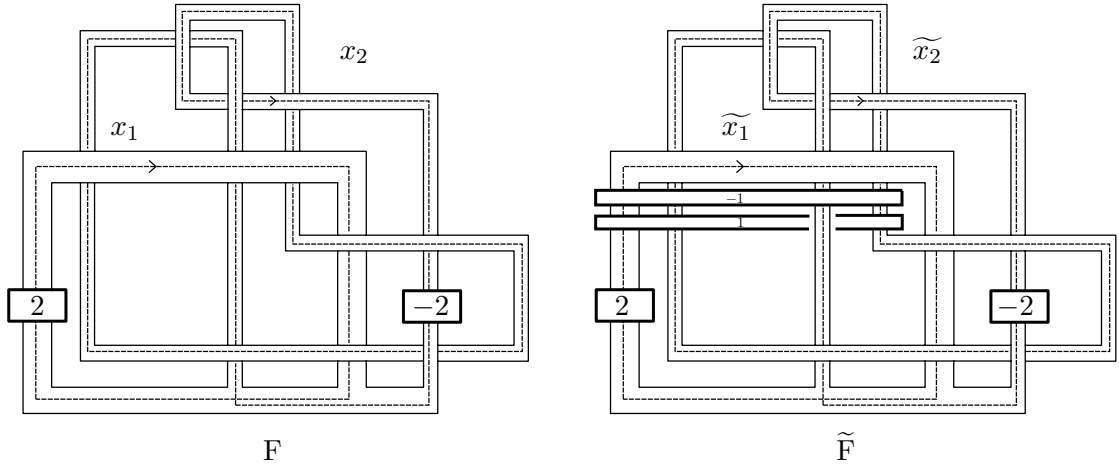

 Let $\widetilde{\F}$ be the Seifert surface for $\R_1$ obtained by doing a $1$-twist Annulus modification on $\phi_\A$ from $\F$ (see Figure $14$). Let $\widetilde{x_1}$ and $\widetilde{x_2}$ be the cores of bands of $\widetilde{\F}$. Then $\{[\widetilde{x_1}],[\widetilde{x_2}]\}$ is a basis for $H_1(\widetilde{\F})$. The Seifert matrix with respect to $\{[\widetilde{x_1}],[\widetilde{x_2}]\}$ is $\widetilde{M}=\begin{pmatrix}
  2 & 0 \\
  -1 & -1
 \end{pmatrix}$, where $\widetilde{M}=(\widetilde{m_{i,j}})=lx(\widetilde{{x_i}},\widetilde{{x_j}}^+)$ and $\widetilde{{x_j}}^+$ is push off of $\widetilde{{x_j}}$ in positive direction. This implies that the derivative curves for $\R_1$ are $\gamma_1$ and $\gamma_2$ where $[\gamma_1]=[\widetilde{x_1}]+[\widetilde{x_2}]$ and $[\gamma_2]=[\widetilde{x_1}]-2[\widetilde{x_2}]$, shown in Figure $15$. 
 
  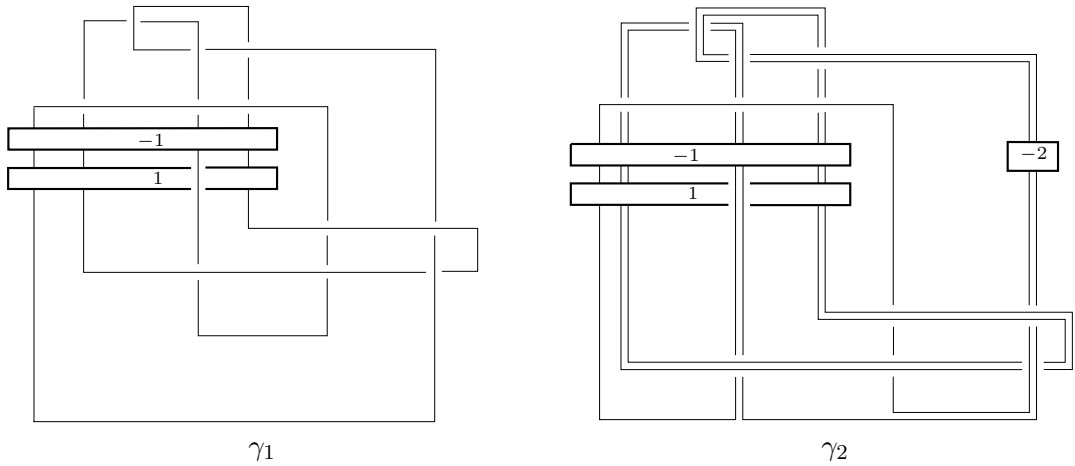
\begin{figure}[h]
\centering
\includegraphics[width=6in]{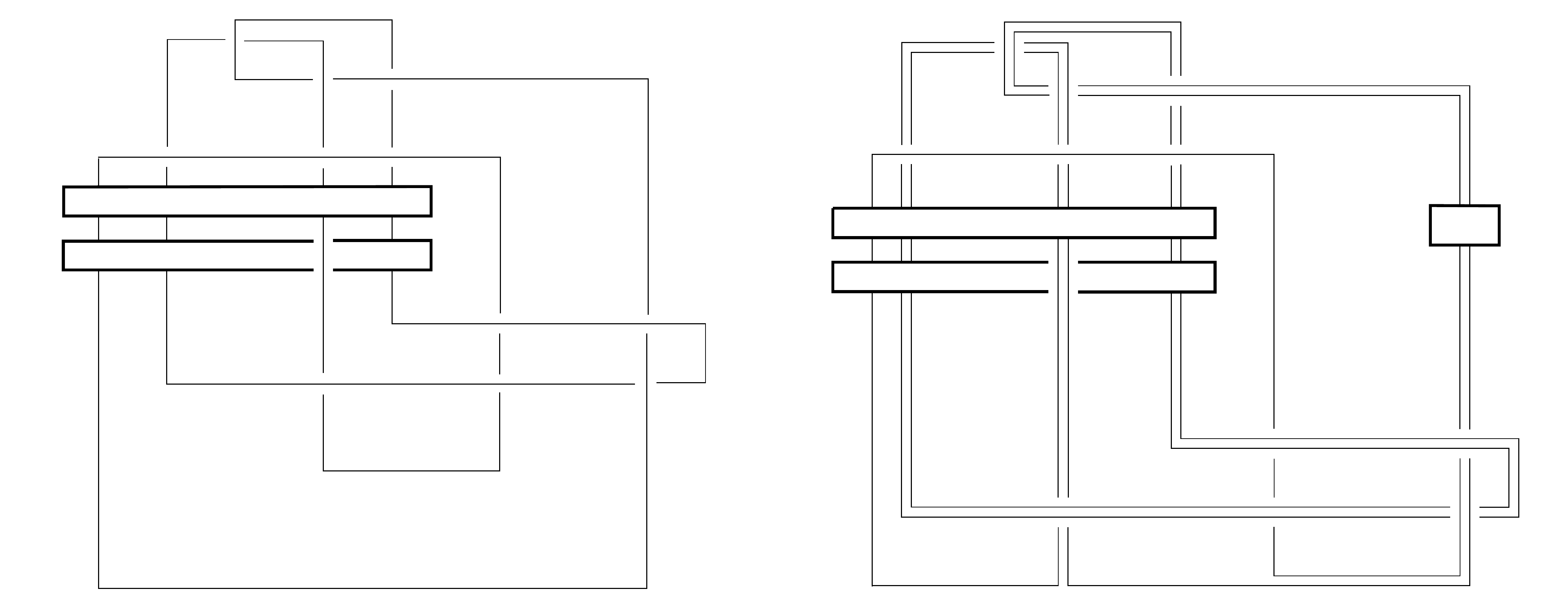}
  \put(-4.5,-0.1){$\gamma_1$} 
  \put(-1.5,-0.1){$\gamma_2$} 
  \put(-5.08,1.52){\tiny{$-1$}}  
  \put(-5,1.32){\tiny{$1$}}  
    \put(-2.28,1.44){\tiny{$-1$}}  
  \put(-2.2,1.24){\tiny{$1$}}  
  \put(-0.46,1.45){\tiny{$-2$}}  
\caption{Non-slice derivatives of $\R_1$}
\end{figure}
 
 We calculate the Alexander polynomial for each derivative curves, $\Delta_{\gamma_1}(t)=4-9t+4t^2$, and $\Delta_{\gamma_2}(t)=1+t+3t^2-11t^3+3t^4+t^5+t^6$. This implies that $\gamma_1$ and $\gamma_2$ are not smoothly slice knot since $\Delta_{\gamma_1}(-1)=\Delta_{\gamma_2}(-1)=17$ which is not a square of an odd prime. Note also that this implies that the curves $\gamma_1$ and $\gamma_2$ are not even algebraically slice.
\end{proof}

\section{Application 2 : Application related to an annulus twist}\label{Application 2}
We will first recall the definition of an oriented annulus presentation of a knot which was first introduced by Osoinach in \cite{Os06}. Detailed discussion could be found in \cite{AJOT13} or \cite{AT14}.

Let $\phi_B : \C \times [0, 1] \hookrightarrow \CCC$ be a smooth embedding of an annulus with $\text{Im}(\phi_B)= B$, and let $c$ be a ${\varepsilon}$ framed unknot in $\CCC$ where $\varepsilon \in \{+1 , -1\}$. These are described in the Figure $16$. Let $\phi_b : [0, 1] \times [0, 1] \hookrightarrow \CCC$ be a smooth embedding of a band with following properties:

\begin{itemize}
\item $\text{Im}(\phi_b)= b$
\item $b \cap \partial B = \text{Im}(\left. \phi_b \right |_{\partial [0,1] \times [0,1]})$ 
\item $b \cap \mathring{B}$ only has ribbon singularities.
\item $b \cap c = \emptyset$
\item $B \cup b$ is orientable.
\end{itemize}

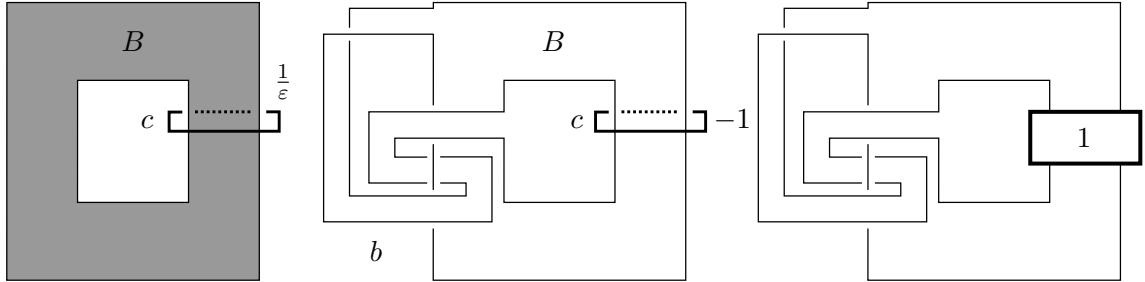
\begin{figure}[h]
\centering
\includegraphics[width=6.1in]{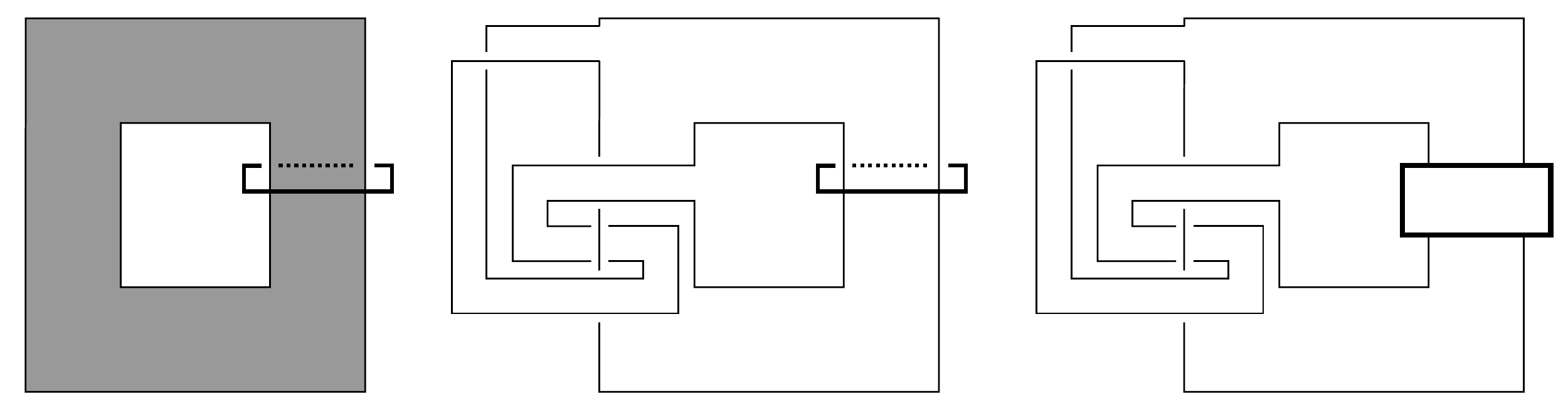}
  \put(-5.4,1.3){$B$}  
   \put(-3.2,1.3){$B$} 
  \put(-5.3,0.9){$c$} 
  \put(-4.6,1.1){$\frac{1}{\varepsilon}$} 
  \put(-4.1,0.2){$b$} 
  \put(-3.05,0.9){$c$} 
  \put(-2.3,0.9){$-1$} 
  \put(-0.4,0.8){$1$} 
\caption{The $8_{20}$ knot admits an oriented annulus presentation $(B,b,-1)$}
\end{figure}

Then we say a knot $\K$ admits an oriented annulus presentation $(B,b,\varepsilon)$ if $(\partial B - b \cap \partial B) \cup \text{Im}(\left. \phi_b \right |_{[0,1] \times \partial [0,1]})$ is isotopic to $\K$ after ${\varepsilon}$ Dehn surgery on $c$. The right side of Figure $16$ shows that the $8_{20}$ knot admits an oriented annulus presentation.

Suppose a knot $\K$ has an annulus presentation $(B,b,\varepsilon)$ and let $B'$ be an annulus which is obtained by slightly pushing $B$ into the interior of $B$, which is described in the Figure $17$. Let $\partial B'=\eta_1 \cup -\eta_2$, and $n$ be some integer, then $\frac{-n\varepsilon+1}{n}$ Dehn surgery on $\eta_1$ and $\frac{-n\varepsilon-1}{n}$ Dehn surgery on $\eta_2$ is called $n$-fold annulus twist on $\K$ defined by Osoinach in \cite{Os06} and the resulting knot will be denoted as $\K_{(B,b,\varepsilon),n}$. Note $lk(\eta_1, \eta_2)=-\varepsilon$.

\begin{figure}[h]
\centering
\includegraphics[width=6.1in]{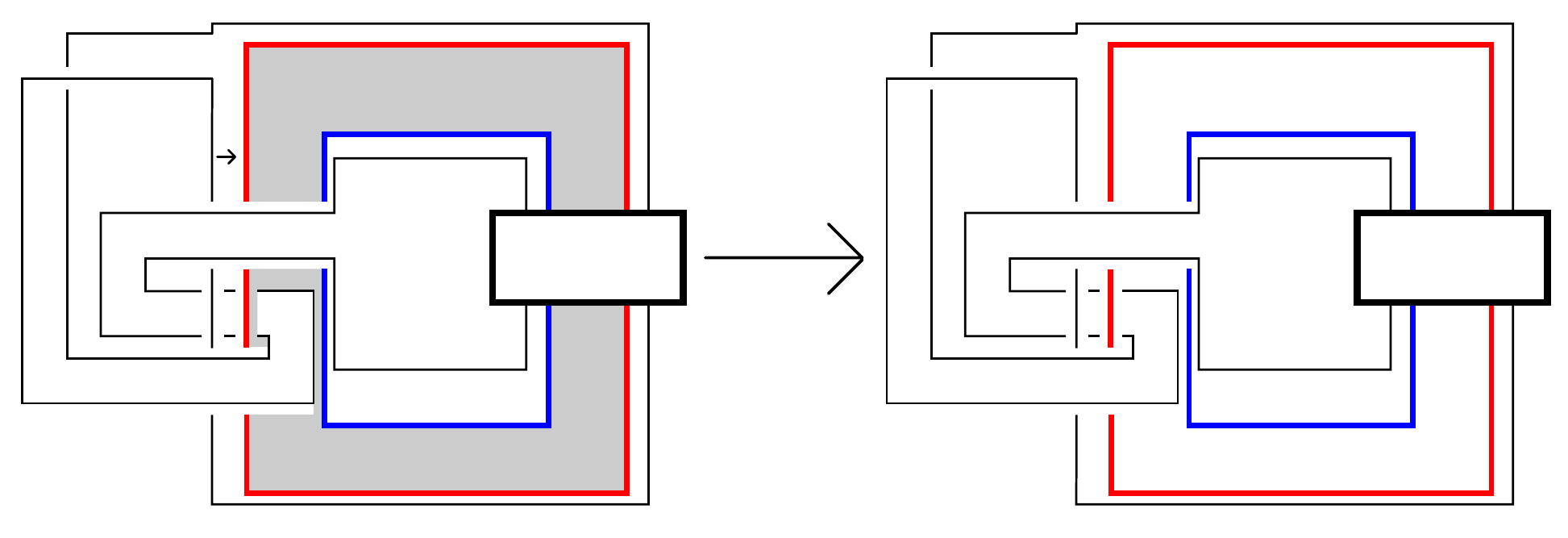}
  \put(-4.3,1.7){$B'$}  
   \put(-4.85,1.82){$\eta_1$} 
      \put(-4.85,1.62){$\eta_2$} 
  \put(-3.85,1.05){$1$} 
  \put(-0.5,1.05){$1$} 
     \put(-1.5,1.82){\tiny{$\frac{n+1}{n}$}} 
      \put(-1.5,1.62){\tiny{$\frac{n-1}{n}$}} 
\caption{$B'$ on the left, and $n$-fold annulus twist on $B'$ on the right for the $8_{20}$ knot}
\end{figure}

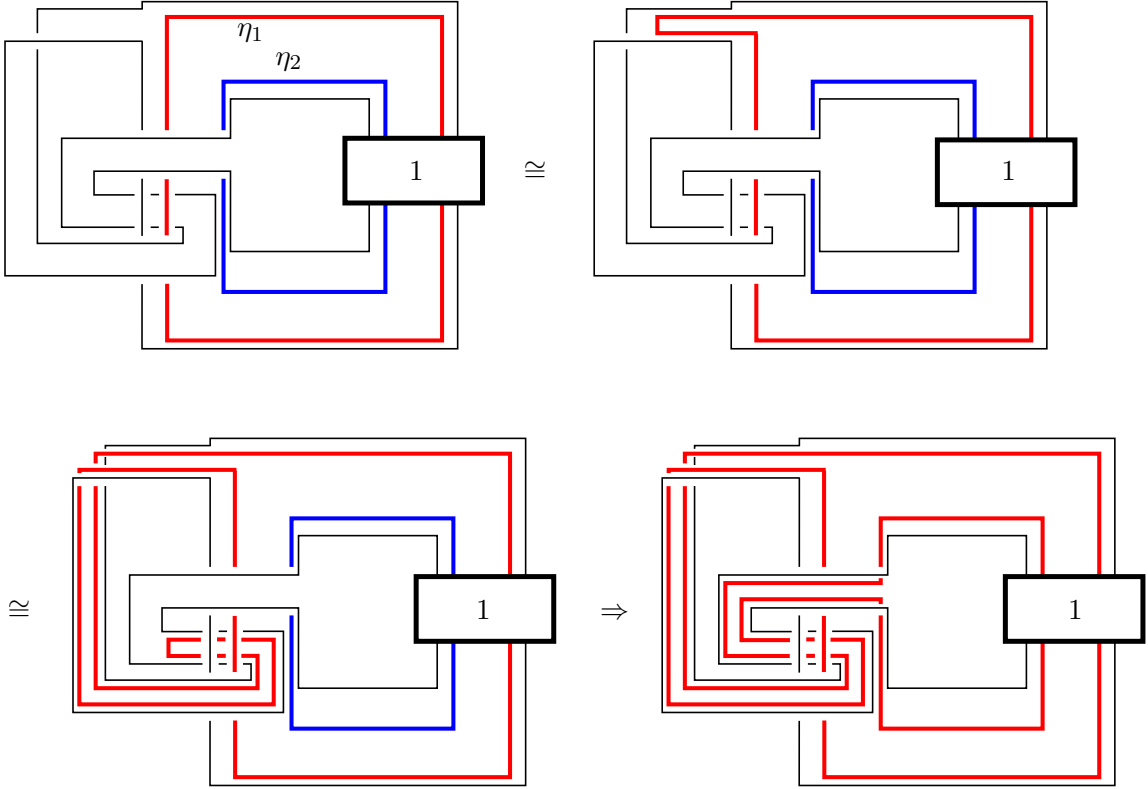
\begin{figure}[h]
\centering
\includegraphics[width=6.1in]{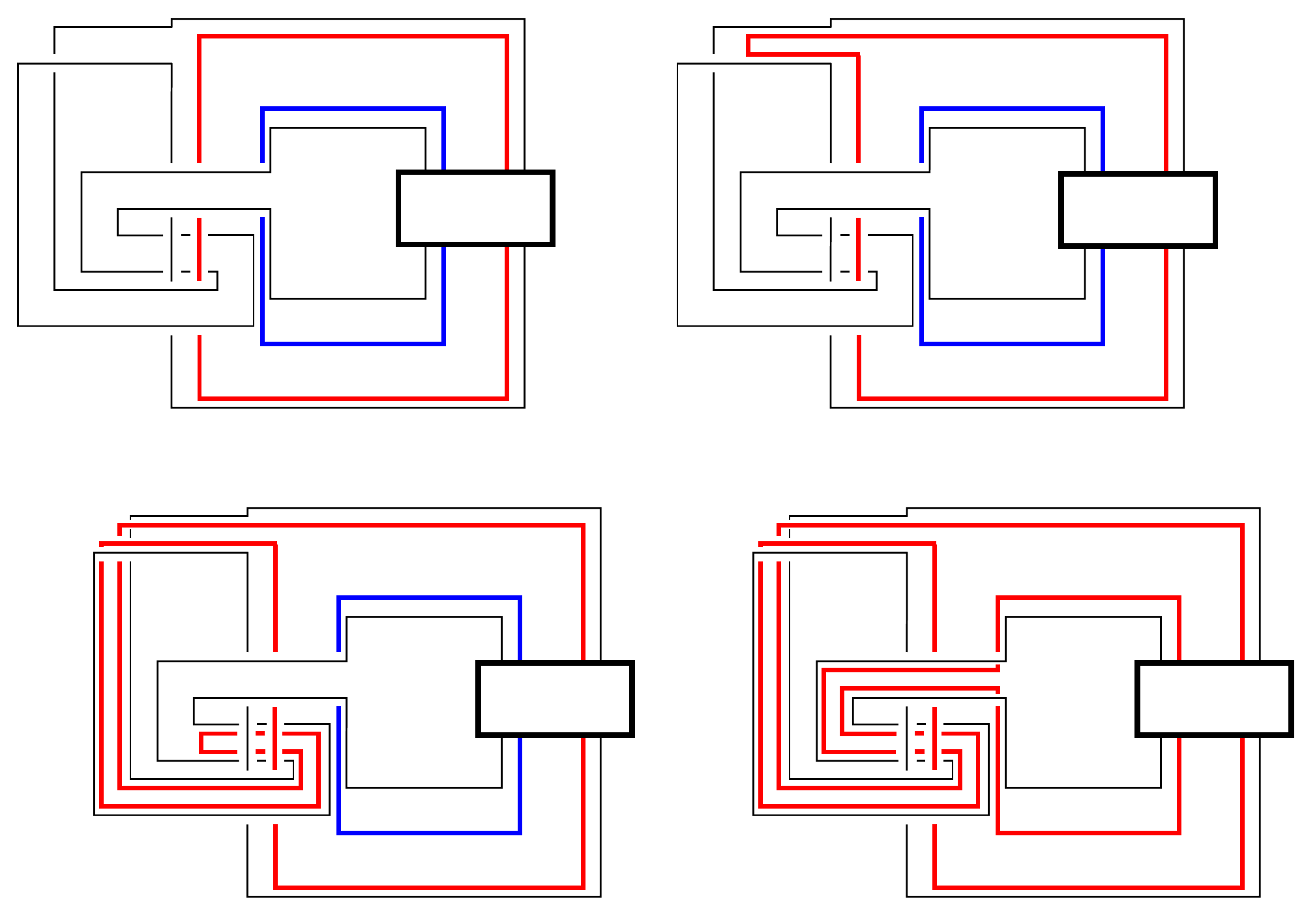}
  \put(-3.9,3.3){$1$}  
    \put(-4.8,4.05){$\eta_1$}  
      \put(-4.6,3.9){$\eta_2$}  
  \put(-0.8,3.3){$1$}  
  \put(-3.55,1){$1$}  
  \put(-0.45,1){$1$}
  \put(-3.3,3.3){$\cong$}
  \put(-6,1){$\cong$}
  \put(-2.9,1){$\Rightarrow$}
\caption{Adding a band between $\eta_1$ and $\eta_2$. The band is following the band $b$.}
\end{figure}

In \cite[Thm 2.3]{Os06} Osoinach showed that $0$ framed Dehn surgery on $\K$ is homeomorphic to $0$ framed Dehn surgery on $\K_{(B,b,\varepsilon),n}$ for any integer $n$. In particular by \cite[Proposition 1.2]{CFHeHo13} if $\K$ is smoothly slice then $\K_{(B,b,\varepsilon),n}$ is exotically slice, which was observed also in \cite{AJOT13}. In this paper we will reprove this statement with slightly stronger assumption using annulus modification without using \cite[Thm 2.3]{Os06}.

\begin{proposition} If $\K$ is a ribbon knot with the annulus presentation $(B,b,\varepsilon)$. Then $\K_{(B,b,\varepsilon),n}$ is exotically slice.
\end{proposition}
\begin{proof}
We will use $B'$, $\eta_1$, and $\eta_2$ described above and in Figure $17$. Note that by Theorem $2.1$ it will be enough to show that there exists a smooth proper embedding $\phi_\A : \C \times [0, 1] \hookrightarrow \B$ such that $\text{Im}(\left. \phi_\A \right |_{\C \times \{ 0 \}})= \eta_1$, $\text{Im}(\left. \phi_\A \right |_{\C \times \{ 1 \}})= \eta_2$ and $(\{ \eta_1, \eta_2 \}, \phi_\A)$ is $(-\varepsilon)$-nice. We can find $\phi_\A$ by performing a band sum as in Figure $18$.

The resulting link after performing a band sum is the $(2,0)$ cable of $\K$ which is a slice link, so we can cap off each component. Then we obtain a slice disk for $\K$ and a smooth proper embedding of an annulus $\phi_\A$ which is disjoint from the slice disk. This guarantees the condition $(1)$ and $(2)$ from the Definition 2.1.

Further, we assumed that $\K$ is a ribbon knot. Then by construction of the annulus $\A$ we know that $\A$ does not contain any local maximum. This implies we have a  surjective map induced by inclusion: $i_* : \pi_1(\CCC - \text{N}(\{ \eta_1, \eta_2 \})) \twoheadrightarrow \pi_1(\B-\N)$. Since $\pi_1(\CCC - \text{N}(\{ \eta_1, \eta_2 \})) = \mathbb{Z}^2$, $\pi_1(\B-\N)$ is an abelian group. Hence $\pi_1(\B-\N)=H_1(\B-\N)=\mathbb{Z}$ which is generated by $[\mu_1] \in H_1(\B-\N)$ where $\mu_1$ is a meridian of $\eta_1$. Then it is easy to check $[c]$ represents trivial element in $\pi_1(\B-\A)$ hence $(\{ \eta_1, \eta_2 \}, \phi_\A)$ satisfies condition $(3)$ from the Definition 2.1.
This implies that $(\{ \eta_1, \eta_2 \}, \phi_\A)$ is $(-\varepsilon)$-nice and concludes the proof.
\end{proof}

\begin{figure}[h]
\centering
\includegraphics[width=5in]{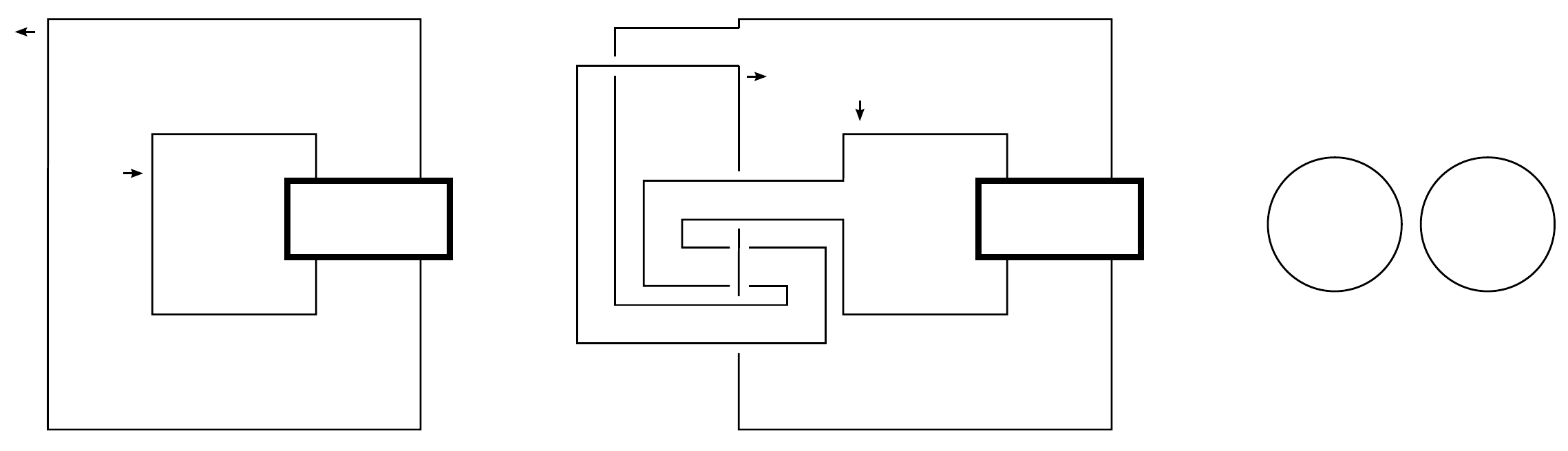}
  \put(-4.5,-0.1){$\{ \eta_1, \eta_2 \}$}  
  \put(-2.1,-0.1){$8_{20}$}  
  \put(-0.7,-0.1){Unlink}  
  \put(-3.86,0.7){$1$}  
  \put(-1.66,0.7){$1$}  
  \put(-3.4,0.7){$\Rightarrow$}
  \put(-1.2,0.7){$\Rightarrow$}
\caption{The second figure is obtained by adding a band between $\eta_1$ and $\eta_2$ as in Figure $18$. The third figure is obtained by doing a ribbon move on $8_{20}$ as in Figure $23$ at the end of this paper. Arrows indicate where the band comes out and get attached to the other component.}
\end{figure}

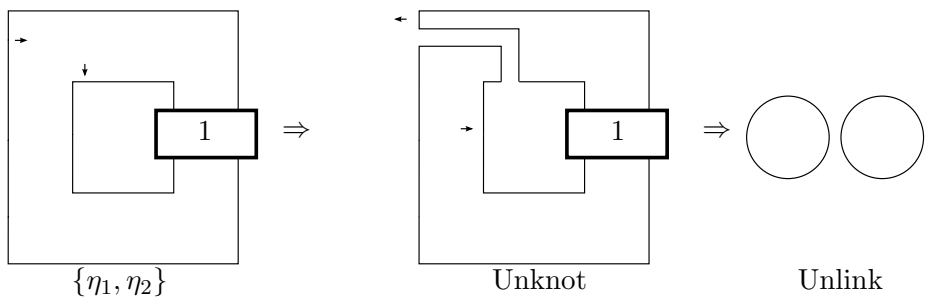
\begin{figure}[h]
\centering
\includegraphics[width=5in]{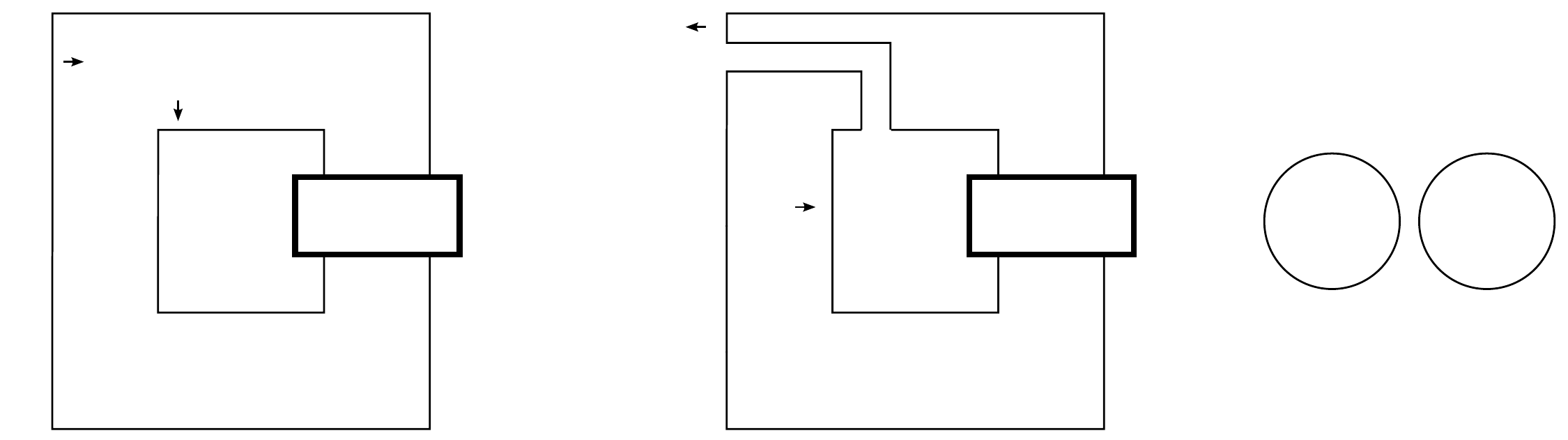}
  \put(-4.5,-0.1){$\{ \eta_1, \eta_2 \}$}  
  \put(-2.3,-0.1){Unknot}  
  \put(-0.7,-0.1){Unlink}  
  \put(-3.84,0.68){$1$}  
  \put(-1.68,0.68){$1$} 
  \put(-3.4,0.7){$\Rightarrow$}
  \put(-1.2,0.7){$\Rightarrow$}
\caption{The second figure is obtained by adding a band between $\eta_1$ and $\eta_2$ as in Figure $23$ at the end of this paper. The third figure is obtained by doing a Ribbon move on an unknot as in Figure $18$. Arrows indicate where the band comes out and get attached to the other component. Note that these two annuli described in Figure $19$ and $20$ are isotopic.}
\end{figure}

Further in \cite{AT14} Abe and Tange showed that if $\K$ is a ribbon knot admitting an annulus presentation $(B,b,\varepsilon)$ where $\varepsilon$ is $1$ or $-1$, then $\K_{(B,b,\varepsilon),n}$ is smoothly slice for any integer $n$. We will reprove this in a very specific case, namely when $\K$ is $8_{20}$ knot, using annulus modifications.

\begin{proposition} Let $\K$ be the $8_{20}$ knot with annulus presentation $(B,b,-1)$ as in Figure $16$. Then $\K_{(B,b,-1),n}$ is smoothly slice for any integer $n$.
\end{proposition}
\begin{proof}

Let $\phi_\A$ be the smooth proper embedding of an annulus described in the proof of Proposition $5.1$. By Theorem $3.1$, it is enough to show that $(\{ \eta_1, \eta_2 \}, \phi_\A)$ is $1$-standard. In other words it is enough to show that $\phi_\A$ is isotopic to $\phi_{\A_1}$ which was described in Section $2$. Note that we can find a ribbon disk for $\K$ by attaching a band as in Figure $23$ at the end of this paper. 

Then in the process of getting $\phi_\A$ we will change the order of band sum by isotopy as in Figure $19$ and $20$. Notice that the knot becomes an unknot after the first band sum (see Figure $20$). Using Scharlemann's corollary in \cite[Corollary page 127]{Sc85} we can isotope the rest of the annulus which is a ribbon disk with two local minima to be a standard disk. This implies that $\phi_\A$ is isotopic to $\phi_{\A_1}$ as needed.
\end{proof}

\section{Annulus modifications on general annuli}\label{Annulus modifications on general annuli}

In this section we will consider more general annuli that we can apply annulus modifications on. We will restrict our attention to the topological category, since we are considering general annuli. More precisely, we will perform an $n$-twist annulus modification to a smooth compact $4$-manifold $\M$ where $\M$ is homeomorphic to $\B$, but not necessarily diffeomorphic to the standard $\B$. When the resulting manifold after an $n$-twist annulus modification is homeomorphic to $\B$, the resulting knot will be exotically slice but not necessarily smoothly slice. We restate the Theorem $2.3$ for the general case.

\begin{theorem}\label{general case} Let $K$ be an exotically slice knot, bounding a smoothly embedded $2$-disk $\D$ in $\M$ where $\M$ is homeomorphic to $\B$. Let $\{ \eta_1, \eta_2 \}$ be an oriented link in $\CCC - K$ and let $\phi_\A : \C \times [0, 1] \hookrightarrow \M$ be a smooth proper embedding of an annulus with $\text{Im}(\left. \phi_\A \right |_{\C \times \{ 0 \}})= \eta_1$, $\text{Im}(\left. \phi_\A \right |_{\C \times \{ 1 \}})= \eta_2$, $\text{Im}(\phi_\A)= \A$, and $l=lk(\eta_1,\eta_2)$.
Suppose $\M_{(\phi_\A,n)}$, the $n$-twist annulus modification on $\M$ at $\phi_\A$, is homeomorphic to $\B$ for an integer $n$ and $\A \cap \D = \emptyset$. Then $\frac{nl+1}{n}$ Dehn surgery on $\eta_1$ followed by $\frac{nl-1}{n}$ Dehn surgery on $\eta_2$ will produce an exotically slice knot $\K_{(\phi_\A,n)} \subseteq \CCC$, where $\K_{(\phi_\A,n)}$ is the image of $\K$ in the new $3$-manifold $\CCC$.\end{theorem}

It is a natural question to ask if there exists a smooth proper embedding of an annulus $\phi_\A : \C \times [0, 1] \hookrightarrow \M$ where $\M_{(\phi_\A,n)}$ is homeomorphic to $\B$ for non-zero $n$, while $\{ \eta_1, \eta_2 \}$ is not isotopic to $L_l$ (see Figure $5$). The following proposition gives us plentiful examples of such smooth proper embedding of an annuli.

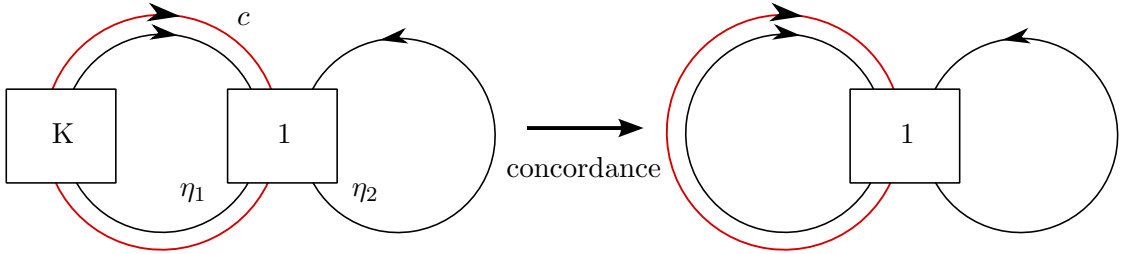
\begin{figure}[h]
\centering
\includegraphics[width=6in]{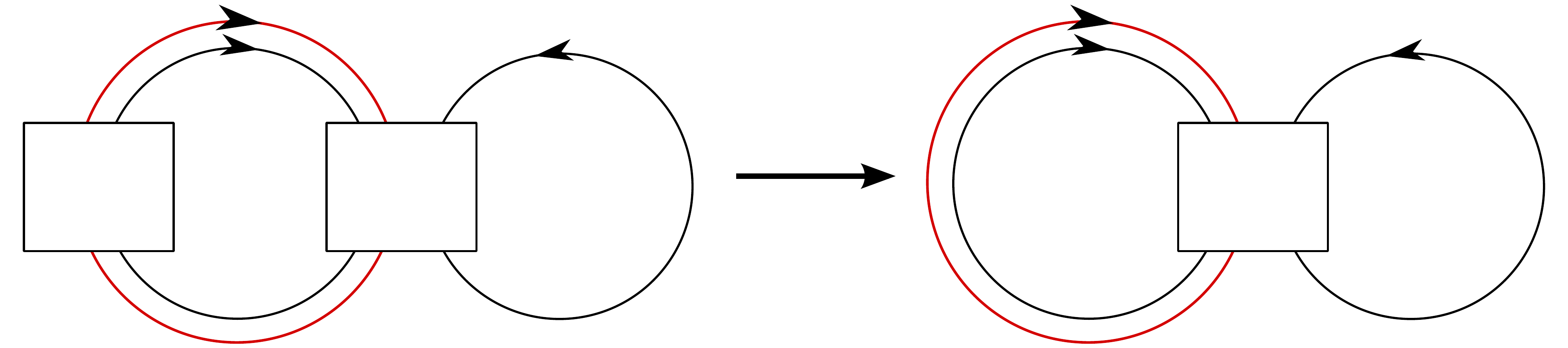}
  \put(-5,0.3){$\eta_1$}  
    \put(-4.1,0.3){$\eta_2$}  
        \put(-4.7,1.2){$c$}  
  \put(-5.67,0.58){$\K$}  
  \put(-4.49,0.58){$1$}  
    \put(-1.23,0.58){$1$}  
    \put(-3.29,0.4){concordance}  
\caption{Positive Hopf link with a knot $\K$ tied up in the first component.}
\end{figure}
\begin{proposition} Let $\{ \eta_1, \eta_2 \}$ be the positive Hopf link with a knot $\K$ tied up in the first component $\eta_1$ (see Figure $21$). If $\K$ is exotically slice in $\M$, then there exists a smooth proper embedding of an annulus $\phi_\A : \C \times [0, 1] \hookrightarrow \M$ with $\text{Im}(\left. \phi_\A \right |_{\C \times \{ 0 \}})= \eta_1$, $\text{Im}(\left. \phi_\A \right |_{\C \times \{ 1 \}})= \eta_2$ where $\M_{(\phi_\A,1)}$ is homeomorphic to $\B$.
\end{proposition}

\begin{proof} Since $\eta_1$ is exotically slice in $\M$, $\{ \eta_1, \eta_2 \}$ is concordant to a positive Hopf link in $\M$. Hence, we can simply cap off the concordance with an annulus that positive Hopf link bounds in $\CCC$ to achieve a smooth proper embedding of an annulus $\phi_\A : \C \times [0, 1] \hookrightarrow \M$ with $\text{Im}(\left. \phi_\A \right |_{\C \times \{ 0 \}})= \eta_1$, $\text{Im}(\left. \phi_\A \right |_{\C \times \{ 1 \}})= \eta_2$. Now we need to show $\M_{(\phi_\A,1)}$ is homeomorphic to $\B$.

It is easy to check that $\M_{(\phi_\A,1)}$ is a homology $\B$ by using a Mayer-Vietoris sequence. In order to see that $\M_{(\phi_\A,1)}$ is simply connected we can use the same argument from the Remark $3.2$. By using a concordance described in Figure $21$, the curve $c$ from Figure $21$ bounds a smoothly embedded disk in $\M-\N$ for the same reason from Remark $3.2$ (see Figure $8$). Therefore, $[c]$ represents a trivial element in $\pi_1(\M-\N)$ which implies that $\pi_1(\M_{(\phi_\A,1)}) = \{ id \}$ by applying the same argument from the proof of Theorem $2.3$.

Lastly, we see that on the boundary we have $2$ Dehn surgery on $\eta_1$ followed by $0$ Dehn surgery on $\eta_2$. Then we can use $\eta_2$ as a helper circle to undo crossings of $\eta_1$ to get a positive Hopf link with coefficient $2$ on one component and $0$ on the other component. Hence we have that $\partial \M_{(\phi_\A,1)}$ is $\CCC$. Then again by the 4-dimensional topological Poincar\'e conjecture we can conclude that $\M_{(\phi_\A,1)}$ is homeomorphic to $\B$ \cite[Theorem $1.6$]{Fr84}.
\end{proof}

Notice that for a link $\{ \eta_1, \eta_2 \}$ isotopic to the negative Hopf link with a knot $\K$ tied up in the first component $\eta_1$, we can also apply the same argument if $\K$ is exotically slice in $\M$. The only difference is that the linking number of $\eta_1$ and $\eta_2$ is negative one. Hence we have a smooth proper embedding of an annulus $\phi_\A : \C \times [0, 1] \hookrightarrow \M$ with $\text{Im}(\left. \phi_\A \right |_{\C \times \{ 0 \}})= \eta_1$ and $\text{Im}(\left. \phi_\A \right |_{\C \times \{ 1 \}})= \eta_2$, where $\M_{(\phi_\A,-1)}$ is homeomorphic to $\B$. In fact there are more examples of such links. For instance, let $\{ \eta_1, \eta_2 \}$ be the Hopf link with a knot $\K$ tied up in the first component $\eta_1$ and assume $\K$ is concordant to an unknotting number one knot. Then by the similar argument from the proof of Proposition $6.2$, it is easy to see that there exists a smooth proper embedding of an annulus $\phi_\A : \C \times [0, 1] \hookrightarrow \M$ with $\text{Im}(\left. \phi_\A \right |_{\C \times \{ 0 \}})= \eta_1$ and $\text{Im}(\left. \phi_\A \right |_{\C \times \{ 1 \}})= \eta_2$, where either $\M_{(\phi_\A,1)}$ or $\M_{(\phi_\A,-1)}$ is homeomorphic to $\B$. By using these general annuli we have the following theorem, which tells us that any exotically slice knot can be obtained by the image of the unknot in the boundary of a smooth $4$-manifold homeomorphic to $\B$ after an annulus modification.

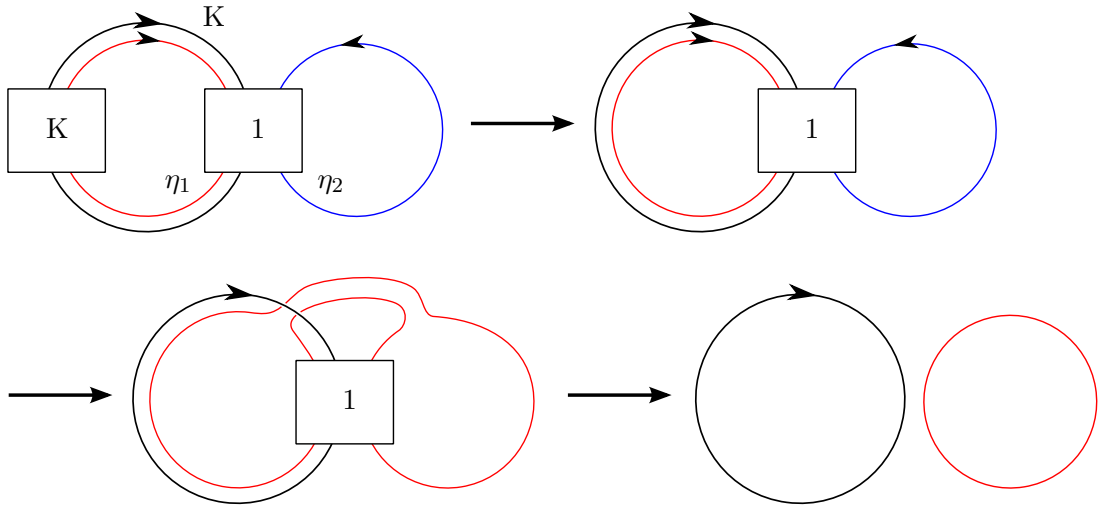
\begin{figure}[h]
\centering
\includegraphics[width=6in]{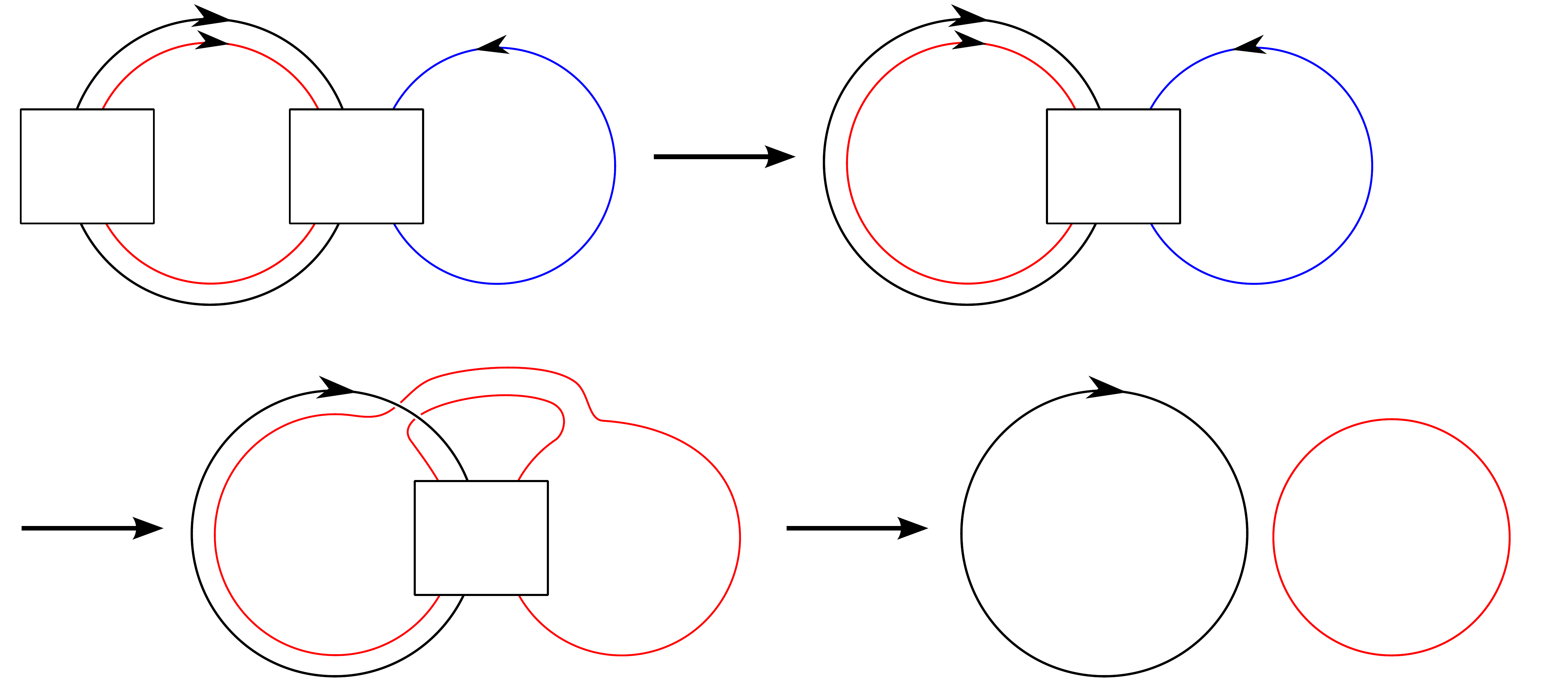}
  \put(-5.72,1.96){$\K$}  
  \put(-4.65,1.96){$1$}  
  \put(-1.75,1.96){$1$}  
    \put(-4.17,0.54){$1$} 
      \put(-4.9,2.55){$\K$}  
        \put(-5.1,1.7){$\eta_1$}  
	\put(-4.3,1.7){$\eta_2$}  
\caption{Smoothly embedded $2$-disk $\D$ and a smooth proper embedding of an annulus $\phi_{\A}$. The second figure is obtained by a concordance. The third figure is obtained by adding a band between the red curve and the blue curve. The fourth figure is obtained by isotopy of the red curve.}
\end{figure}

\begin{theorem}\label{general case} Let $K$ be an exotically slice knot in $\M$ where $\M$ is homeomorphic to $\B$. Then there exists a $\M'$ which is homeomorphic to $\B$, the unknot $U \subseteq \partial{{\M}'} = \CCC$, and a smooth proper embedding of an annulus $\phi_{\A'} : \C \times [0, 1] \hookrightarrow \M'$ with following properties:
\begin{itemize}
\item $\text{Im}(\left. \phi_{\A'} \right |_{\C \times \{ 0 \}})= \eta'_1$, $\text{Im}(\left. \phi_{\A'} \right |_{\C \times \{ 1 \}})= \eta'_2$, $\text{Im}(\phi_{\A'})= \A'$, and $1=lk(\eta'_1,\eta'_2)$.
\item $\A' \cap D' = \emptyset$ where $D'$ is a smoothly embedded $2$-disk in $\M'$ which bounds the unknot $U$.
\item The $-1$-twist annulus modification on $\M'$ at $\phi_{\A'}$, $\M'_{(\phi_{\A'},-1)}$, is homeomorphic to $\M$.
\item The image of the unknot $U$ after the $-1$-twist annulus modification on $\M'$ at $\phi_{\A'}$, $U_{(\phi_{\A'},-1)}$, is $\K \subseteq \CCC = \partial{\M}$.
\end{itemize}
\end{theorem}

\begin{proof} Notice that it is enough to show that there exist a smoothly embedded $2$-disk $\D$ in $\M$ which bounds $\K$ and a smooth proper embedding of an annulus $\phi_{\A} : \C \times [0, 1] \hookrightarrow \M$ with $\text{Im}(\left. \phi_{\A} \right |_{\C \times \{ 0 \}})= \eta_1$, $\text{Im}(\left. \phi_{\A} \right |_{\C \times \{ 1 \}})= \eta_2$, $\text{Im}(\phi_{\A})= \A$, $\A \cap \D = \emptyset$ and $1=lk(\eta_1,\eta_2)$, so that $\M_{(\phi_{\A},1)}$ is homeomorphic to $\B$ and $K_{(\phi_{\A},1)}$ is the unknot, since we can simply perform the $-1$-twist annulus modification on $\M_{(\phi_{\A},1)}$ with the same annulus. We will use a smoothly embedded $2$-disk $\D$ and a smooth proper embedding of an annulus $\phi_{\A}$ described in Figure $22$. By Proposition $6.2$, $\M_{(\phi_{\A},1)}$ is homeomorphic to $\B$. In addition by performing handle slides, isotopies and Rolfsen twists, we see that $K_{(\phi_{\A},1)}$ is the unknot as needed (see Figure $24$ at the end of this paper).
\end{proof}

\bibliographystyle{alpha}
\bibliography{knotbib}

\newpage
\begin{figure}[h]
\centering
\includegraphics[width=6.1in]{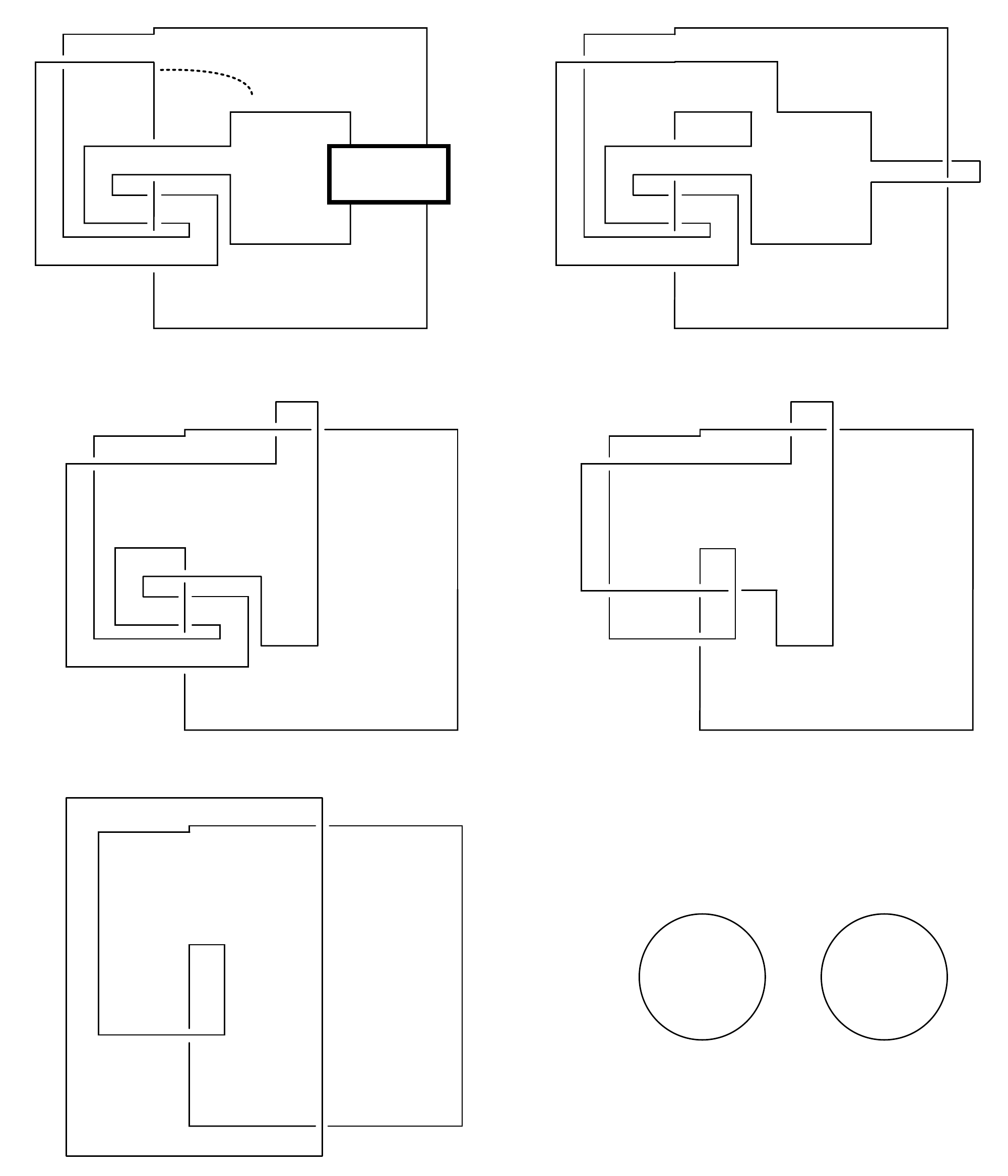}
  \put(-3.77,6){$1$}  
  \put(-3.1,6){$\Rightarrow$}
    \put(-6,3.5){$\cong$}
      \put(-3,3.5){$\cong$}
        \put(-6,1.1){$\cong$}        
        \put(-3,1.1){$\cong$}
\caption{Performing a ribbon move on $8_{20}$ gives a two component unlink. The second figure is obtained by performing a ribbon move along the dotted line. The rest of the figures are obtained by isotopies.}
\end{figure}

\newpage
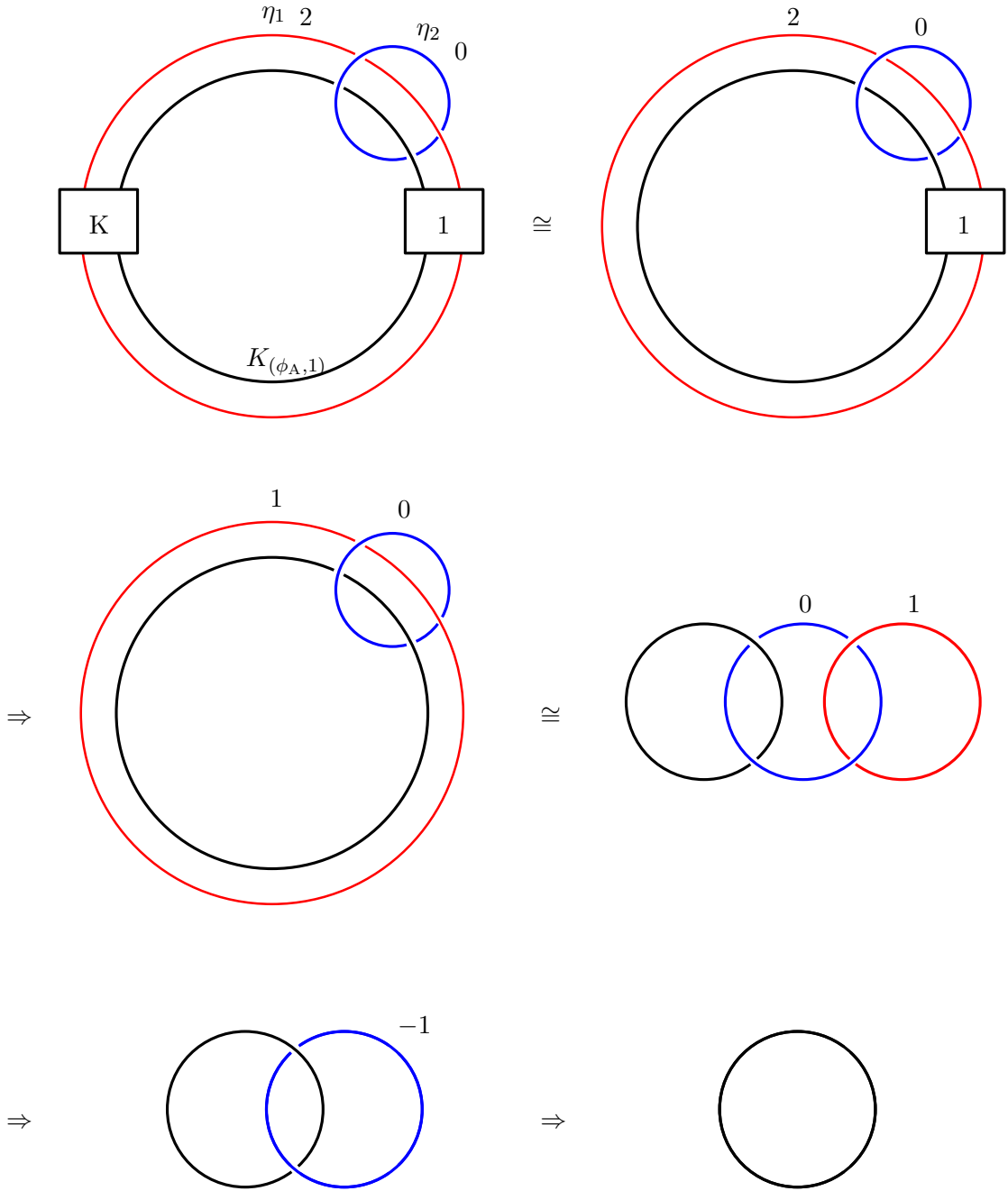
\begin{figure}[h]
\centering
\includegraphics[width=5.5in]{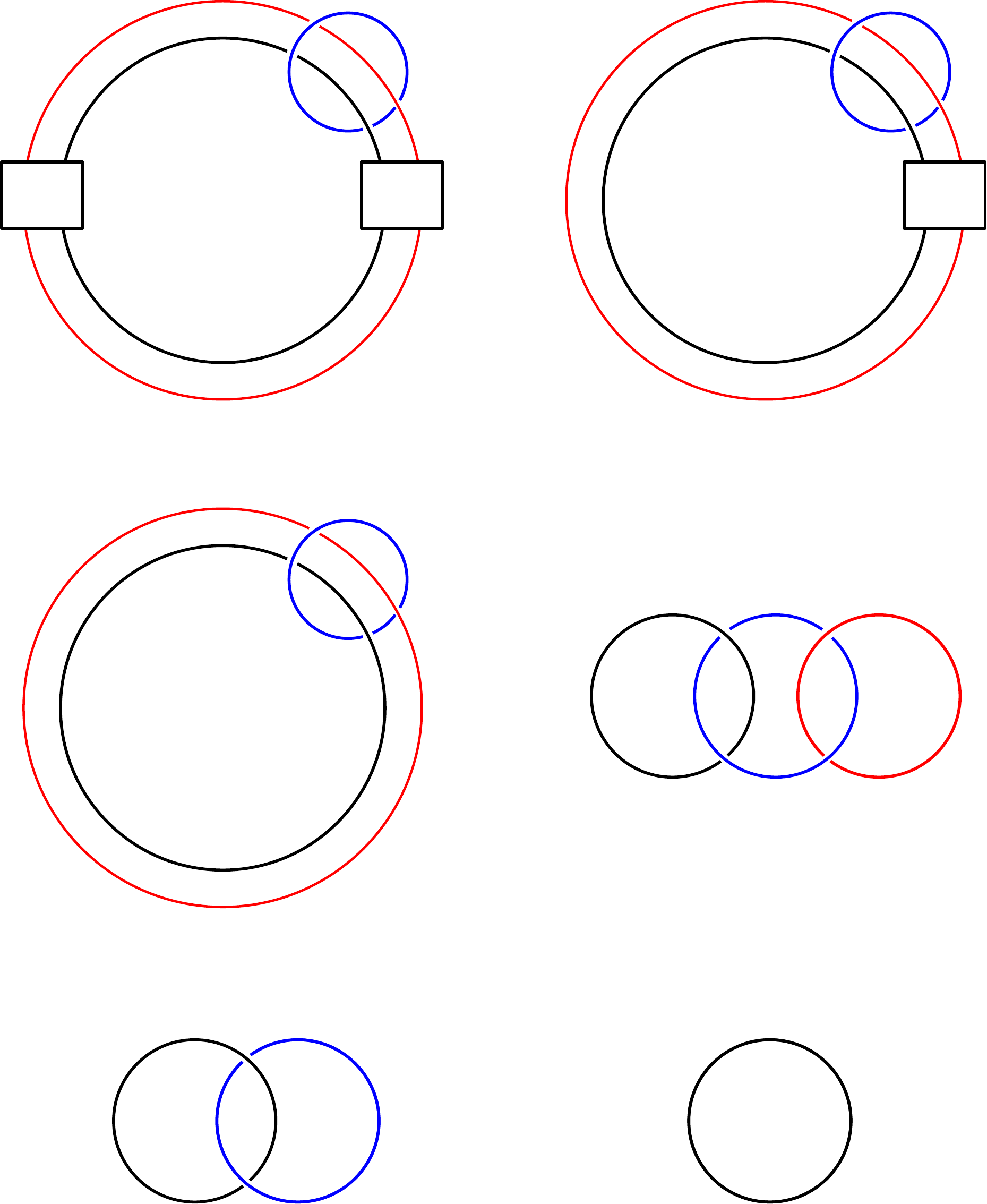}
  \put(-3.3,5.55){$1$}
    \put(-5.32,5.55){$\K$}
      \put(-0.28,5.55){$1$}
	\put(-4.32,6.8){$\eta_1$}
	\put(-4.1,6.76){$2$}
	\put(-3.42,6.7){$\eta_2$}
	\put(-3.2,6.56){$0$}
	\put(-4.46,4.78){ $K_{(\phi_{\A},1)}$}
	\put(-1.27,6.76){$2$}
	\put(-0.529,6.7){$0$}
	\put(-4.27,3.96){$1$}
	\put(-3.529,3.9){$0$}
	\put(-1.2,3.35){$0$}
	\put(-0.57,3.35){$1$}
		\put(-3.529,0.9){$-1$}
  \put(-2.75,5.55){$\cong$}
    \put(-5.8,2.7){$\Rightarrow$}
      \put(-2.7,2.7){$\cong$}
        \put(-5.8,0.35){$\Rightarrow$}        
        \put(-2.7,0.35){$\Rightarrow$}
\caption{The first figure describes $K_{(\phi_{\A},1)}$. The second figure is obtained by performing handle slides. The rest of the figures are obtained by performing Rolfsen twists and isotopies.}
\end{figure}

\end{document}